\documentclass{amsart}

\usepackage{graphicx}
\usepackage{libertine}
\usepackage{wasysym}

\newtheorem{theorem}{Theorem}[section]
\newtheorem{lemma}[theorem]{Lemma}
\newtheorem{corollary}[theorem]{Corollary}
\newtheorem{prop}[theorem]{Proposition}

\newtheorem{notation}[theorem]{Notation}

\theoremstyle{definition}
\newtheorem{definition}[theorem]{Definition}

\theoremstyle{remark}
\newtheorem{remark}[theorem]{Remark}

\numberwithin{equation}{section}

\newcommand*{\QED}{\hfill\ensuremath{\square}}

\usepackage[top=0.75in, bottom=0.75in, left=0.75in, right=0.75in]{geometry}

\usepackage{mathrsfs}

\usepackage[english]{babel}
\usepackage[utf8x]{inputenc}
\usepackage{amsmath}
\usepackage{graphicx}
\usepackage{amssymb}
\usepackage{amsthm}
\usepackage{fancyhdr}
\usepackage{amsmath}
\usepackage{amsxtra}
\usepackage{amscd}
\usepackage{amsthm}
\usepackage{amsfonts}
\usepackage{amssymb}
\usepackage{eucal}
\usepackage[all]{xy}
\usepackage{graphicx}
\usepackage{pifont}
\usepackage{comment}
\usepackage{tikz}
\usetikzlibrary{matrix,arrows}

\newcommand\nc{\newcommand}
\nc{\on}{\operatorname}
\nc{\R}{\mathbb R}
\nc{\C}{\mathbb C}
\nc{\Q}{\mathbb Q}
\nc{\Z}{\mathbb Z}
\nc{\N}{\mathbb N}
\nc{\F}{\mathbb F}
\nc{\Hom}{\on{Hom}}
\nc{\wt}{\widetilde}
\nc{\kernel}{\text{ker}}
\nc{\image}{\text{Im}}
\nc{\sls}{\subsetneq ... \subsetneq}
\nc{\ssn}{\subsetneq}
\nc{\bull}{$\bullet \, \,$}
\nc{\ol}{\overline}
\nc{\short}[3]{0 \longrightarrow #1 \longrightarrow #2 \longrightarrow #3 \longrightarrow 0}
\nc{\pd}[2]{\frac{\partial #1}{\partial #2}}
\nc{\one}{\mathbf{1}}
\nc{\rnc}{\renewcommand}
\nc{\e}{\varepsilon}
\nc{\DMO}{\DeclareMathOperator}
\nc{\dd}{\emph{d}}
\nc{\grad}{\nabla}

\rnc{\leq}{\leqslant}
\rnc{\geq}{\geqslant}
\rnc{\int}{\varint}
\rnc{\d}{\text{d}}

\DeclareMathOperator{\E}{\mathbb{E}}

\DMO{\D}{\text{D}}
\DMO{\pv}{p.v.}

\newenvironment{nouppercase}{%
  \renewcommand{\uppercasenonmath}[1]{}}{}
\begin{document}


\title{\LARGE Local Marchenko-Pastur Law for Random Bipartite Graphs}

\author{\large Kevin Yang$^{\dagger}$}

\begin{nouppercase}
\maketitle
\end{nouppercase}
\begin{center}
\today
\end{center}
\footnote{$\dagger$ Stanford University, Department of Mathematics. Email: kyang95@stanford.edu.}

\begin{abstract}
We study the random matrix ensemble of covariance matrices arising from random $(d_b, d_w)$-regular bipartite graphs on a set of $M$ black vertices and $N$ white vertices, for $d_b \gg \log^4 N$. We simultaneously prove that the Green's functions of these covariance matrices and the adjacency matrices of the underlying graphs agree with the corresponding limiting law (e.g. Marchenko-Pastur law for covariance matrices) down to the optimal scale. This is an improvement from the previously known mesoscopic results. We obtain eigenvector delocalization for the covariance matrix ensemble as consequence, as well as a weak rigidity estimate. 
\end{abstract}

\tableofcontents

\newpage
\section{Introduction}
Let $X_A$ denote the adjacency matrix of a random $(d_b, d_w)$-biregular graph with off-diagonal blocks $A, A^{\ast}$. Here, we assume $A$ is a matrix of size $M \times N$ with $M \geq N$. We define the normalized empirical spectral distributions of $d_w^{-1} A^{\ast}A$ and $d_w^{-1/2} X_A$ to be the following macroscopic random point masses:
\begin{align}
\mu_{A^{\ast}A} \ &= \ \frac{1}{N} \sum_{\lambda \in \sigma(A^{\ast}A)} \ \delta_{d_w^{-1} \lambda}(x), \\
\mu_{X_A} \ &= \ \frac{1}{M+N} \sum_{\lambda \in \sigma(X_A)} \ \delta_{d_w^{-1/2} \lambda}(x).
\end{align}
It is known (see \cite{BS}) that the empirical spectral distribution of the normalized covariance matrix converges almost surely (in the limit $M, N \to \infty$ and $d \to \infty$ at a suitable rate) to the Marchenko-Pastur law with parameter $\gamma := N/M$ given by the following density function:
\begin{align}
\varrho_{\infty}(x) \ \d x \ := \ \varrho_{MP}(x) \ \d x \ = \ \frac{\sqrt{(\lambda_+ - x)(x - \lambda_-)}}{2 \pi \gamma x} \mathbf{1}_{x \in [\lambda_{-}, \lambda_{+}]} \d x, \label{eq:MPlaw}
\end{align}
where we define $\lambda_{\pm} = (1 \pm \sqrt{\gamma})^2$. As noted in \cite{DJ}, this implies that the empirical spectral distribution of the normalized adjacency matrix converges almost surely to a linearization of the Marchenko-Pastur law given by the following density function:
\begin{align}
\varrho(E) \ = \ \begin{cases}
	\frac{\gamma}{(1 + \gamma) \pi |E|} \sqrt{(\lambda_{+} - E^2)(E^2 - \lambda_{-})} & E^2 \in [\lambda_{-}, \lambda_{+}] \\
	0 & E^2 \not\in [\lambda_{-}, \lambda_{+}]
	\end{cases}.
\end{align}
We briefly remark that in the regime $M = N$, the linearized Marchenko-Pastur density agrees exactly with the Wigner semicircle density, which is the limiting density for the empirical spectral distribution of random $d$-regular graphs on $N$ vertices in the limit $N, d \to \infty$ at suitable rates. In this regime, coincidence of the empirical spectral distribution and the limiting Wigner semicircle law was shown for intervals at the optimal scale $N^{-1 + \e}$ in \cite{BKY}. This short-scale result is crucial for understanding eigenvalue gap and correlation statistics and showing universality of eigenvalue statistics for random regular graphs compared to the GOE. 

Moreover, the short-scale result is a drastic improvement from the order 1 result discussed above for biregular bipartite graphs. For these graphs, convergence of the empirical spectral distribution of $d_w^{-1} A^{\ast}A$ to the Marchenko-Pastur law was shown for scales $N^{-\e}$ for sufficiently small $\e > 0$ in \cite{DJ}. The techniques used in this paper included primarily analysis of trees and ballot sequences. This result, however, is far from the optimal scale and is thus far from sufficient for showing universality of eigenvalue statistics. The aim of this paper is remedy this problem and obtain convergence at the optimal scale. Similar to \cite{BKY}, we bypass the analysis of trees and ballot sequences with a combinatorial operator on graphs known as \emph{switchings}, which are ubiquitous throughout graph theory. This will help us resample vertices in a random graph and will be crucial in deriving a tractable self-consistent equation for the Green's function of a random biregular bipartite graph. However, in contrast to \cite{DJ}, we aim to prove convergence at the optimal scale for the ensembles of (normalized) covariance matrices and adjacency matrices simultaneously, transferring between the analysis of each ensemble whenever convenient. In \cite{DJ}, the result for adjacency matrices was derived as a consequence of the result for covariance matrices. To the author's knowledge, this idea is original to this paper. 

For random regular graphs, universality of local bulk eigenvalue correlation statistics was shown in \cite{BHKY}, which required both the local law from \cite{BKY} as a crucial ingredient as well as analysis of the Dyson Brownian Motion in \cite{LY}. In a similar spirit for random biregular bipartite graphs, we prove universality of bulk eigenvalue statistics in \cite{Y2} using the local law result in this paper and an analysis of Dyson Brownian Motion for covariance matrices in \cite{Y3}. Thus, this paper may be viewed as the first in a series of three papers on random covariance matrices resembling the papers \cite{BKY}, \cite{BHKY}, \cite{LY}, and \cite{LSY} which study Wigner matrices. 

Before we proceed with the paper, we remark that as with random regular graphs and Wigner ensembles, the covariance matrix ensembles arising from biregular bipartite graphs is a canonical example of a covariance matrix ensemble whose entries are nontrivially correlated. A wide class of covariance matrices with independent sample data entries was treated in the papers \cite{A}, \cite{BEKYY}, \cite{BKYY}, \cite{NP2}, and \cite{TV}. In these papers, local laws were derived and universality of local eigenvalue correlation statistics were proven assuming moment conditions. Because of the nontrivial correlation structure of the sample data entries and the lack of control of entry-wise moments, these papers and their methods cannot apply to our setting. 
\subsection{Acknowledgements}
The author thanks H.T. Yau and Roland Bauerschmidt for suggesting the problem, referring papers, and answering the author's questions pertaining to random regular graphs. This work was partially funded by a grant from the Harvard College Research Program. This paper was written while the author was a student at Harvard University.
\subsection{Notation}
We adopt the Landau notation for big-Oh notation, and the notation $a \lesssim b$. We establish the notation $[[a, b]] := [a,b] \cap \Z$. We let $[E]$ denote the underlying vertex set of a graph $E$. For vertices $v, v' \in E$, we let $vv'$ denote the edge in $E$ containing $v$ and $v'$. For a real symmetric matrix $H$, we let $\sigma(H)$ denote its (real) spectrum.
%
%
%
\section{Underlying Model and Main Results}
We briefly introduce the underlying graph model consisting of bipartite graphs on a prescribed vertex set. 
\begin{definition} \label{definition:bipartitegraph}
Suppose $\mathscr{V} = \{ 1_b, 2_b, \ldots, M_b, 1_w, \ldots, N_w \}$ is a set of labeled vertices, and suppose $E$ is a simple graph on $\mathscr{V}$. We say the graph $E$ is \emph{bipartite} with respect to the vertex sets $(\mathscr{V}, V_b, V_w)$ if $\mathscr{V}$ admits the following decomposition:
\begin{align}
\mathscr{V} \ = \ \left\{ 1_b, 2_b, \ldots, M_b \right\} \bigcup \left\{ 1_w, 2_w, \ldots, N_w \right\} \ =: \ V_b \cup V_w,
\end{align}
such that for any vertices $v_i, v_j \in V_b$ and $v_k, v_\ell \in V_w$, the edges $v_i v_j$ and $v_k v_\ell$ are not contained in $E$. 

Moreover, for fixed integers $d_b, d_w > 0$, we say that a bipartite graph $E$ is $(d_b, d_w)$-\emph{regular} if each $v \in V_b$ has $d_b$ neighbors and if each $w \in V_w$ has $d_w$ neighbors.
\end{definition}
\begin{remark}
For the remainder of this paper, we will refer to $V_b$ as the set of \emph{black} vertices and $V_w$ as the set of \emph{white} vertices. Moreover, we will refer to a $(d_b, d_w)$-regular graph simply as a \emph{biregular} graph, if the parameters $d_b, d_w$ are assumed. In particular, when referring to a biregular graph we assume a bipartite structure. Lastly, the set of $(d_b, d_w)$-regular graphs on the vertex sets $(\mathscr{V}, V_b, V_w)$ will be denoted by $\Omega$, where we suppress the dependence of the parameters $M, N, d_b, d_w$ without the risk of confusion.
\end{remark}
We now record the following identity which follows from counting the total number of edges in a biregular graph $E$:
\begin{align}
Md_b \ = \ N d_w, \label{eq:countedges}
\end{align}
where $M = M_b$ and $N = N_w$. We retain this notation for $M, N$ for the remainder of the paper. 
\subsection{The Random Matrix Ensemble}
We now introduce a modification of the random matrix ensemble studied in \cite{DJ}, retaining the notation used in the introduction of this paper. We first note the adjacency matrix of a biregular graph is a block matrix with vanishing diagonal, i.e. has the following algebraic form:
\begin{align}
X_A \ = \ \begin{pmatrix} 0 & A \\ A^\ast & 0 \end{pmatrix},
\end{align}
where $A$ is a matrix of size $M \times N$. By the biregular assumption of the underlying graph, the matrix $X_A$ exhibits the following eigenvalue-eigenvector pair:
\begin{align}
\lambda_{\max} \ = \ \sqrt{d_b d_w}, \quad \mathbf{e}_{\max} \ = \ \begin{pmatrix} \mathbf{e}_b \\ \sqrt{\alpha} \mathbf{e}_w \end{pmatrix},
\end{align}
where $\mathbf{e}_b$ and $\mathbf{e}_w$ are constant $\ell^2$-normalized vectors of dimension $M$ and $N$ respectively. By the Perron-Frobenius theorem, the eigenvalue $\lambda_{\max}$ is simple. 

Using ideas from \cite{BKY}, the matrix ensemble of interest is the ensemble $\mathscr{X} = \mathscr{X}(M, N, d_b, d_w)$ of \emph{normalized} adjacency matrices given by
\begin{align}
X \ = \ \begin{pmatrix} 0 & H \\ H^\ast & 0 \end{pmatrix}, \quad H \ = \ d_w^{-1/2} \left( A - d_b \mathbf{e}_{b} \mathbf{e}_{w}^\ast \right), \label{eq:normalizedmatrix}.
\end{align}
Because the eigenspace corresponding to $\lambda_{\max}$ is one-dimensional, standard linear algebra implies that upon a normalization factor of $d_w^{-1/2}$, the matrices $X$ and $X_A$ will share the same eigenvalue-eigenvector pairs orthogonal to the eigenspace corresponding to $\lambda_{\max}$. On this maximal eigenspace, the matrix $X$ will exhibit $\mathbf{e}_{\max}$ as an eigenvector corresponding to the eigenvalue $\lambda = 0$. Moreover, as noted in \cite{DJ} the spectrum of $X$ will be compactly supported in a sense we will make shortly make precise.

To complete our discussion of the random matrix ensemble of interest, we first note that $\mathscr{X}$ is clearly in bijection with the set of biregular graphs on the fixed triple $(\mathscr{V}, V_b, V_w)$, which is a finite set for each fixed $M, N, d_b, d_w$. Because $\Omega$ is finite, we may impose the uniform probability measure on it. 

We now define the following fundamental parameters:
\begin{align}
\alpha \ := \ \frac{M}{N} \ = \ \frac{d_w}{d_b}, \quad \gamma \ := \ \frac{1}{\alpha} \ = \ \frac{N}{M}.
\end{align}
where here we use the identity \eqref{eq:countedges}. As in \cite{DJ}, we now impose the following constraints on the parameters $M, N, d_b, d_w$:
\begin{align}
\lim_{M, N \to \infty} \ \alpha \ \geq \ 1. \label{eq:limitratio}
\end{align}
This assumption is not crucial as we may also relabel the vertices $\mathscr{V}$ if $M < N$ in the limit. This assumption will be convenient in our analysis of the spectral statistics, however. This completes our construction of the random matrix ensemble of interest.
\subsection{Random Covariance Matrices}
Strictly speaking, the random matrix ensemble studied in \cite{DJ} was the ensemble $\mathscr{X}_{\ast}$ consisted of the corresponding $N \times N$ covariance matrices:
\begin{align}
X_{\ast} \ = \ H^{\ast} H, \quad H \ = \ d_w^{-1/2}(A - d_b \mathbf{e}_b \mathbf{e}_w^{\ast}).
\end{align}
The ensemble $\mathscr{X}$ introduced in this paper may be realized as a linearization of the ensemble $\mathscr{X}_{\ast}$ of covariance matrices. This will be the upshot of working instead with the matrix ensemble $\mathscr{X}$ whenever convenient, i.e. when studying linear perturbations of the adjacency matrix of a biregular graph. The following result shows that when transferring between the ensembles $\mathscr{X}$ and $\mathscr{X}_{\ast}$, the spectral data is preserved. This result is standard in linear algebra and the analysis of compact operators, but we include it for completeness and since organizational purposes, as the result does not seem to be written precisely and formally in any standard text. 

Before we give the result and its (short) proof, we define the following third matrix ensemble $\mathscr{X}_{\ast,+}$ of $M \times M$ covariance matrices:
\begin{align}
X_{\ast,+} \ = \ H H^{\ast}.
\end{align}
The ensemble $\mathscr{X}_{\ast,+}$ will not play any essential role in our analysis of random matrix ensembles and is included in this paper for the sake of completeness of our results.
\begin{prop} \label{prop:spectralcorrespondence}
Suppose $H$ is a real-valued matrix of size $M \times N$ with $M \geq N$, and suppose $X$ is a block matrix of the following form:
\begin{align}
X \ = \ \begin{pmatrix} 0 & H \\ H^{\ast} & 0 \end{pmatrix}. \label{eq:blockmatrix}
\end{align}
%
\begin{itemize}
\item \emph{(I).} The spectrum of $X$ admits the following decomposition:
\begin{align}
\sigma(X) \ = \ \sigma^{1/2}(H^{\ast} H) \cup \zeta(X),
\end{align}
where $\sigma^{1/2}(H^{\ast} H)$ denotes the pairs of eigenvalues $(\pm \lambda)$ such that $(\pm \lambda)^2$ is an eigenvalue of $H^{\ast} H$. Here, $\zeta(X)$ denotes the set of eigenvalues not in $\sigma^{1/2}(H^{\ast} H)$, all of which are 0.
\item \emph{(II).} The spectrum of $HH^{\ast}$ admits the following decomposition:
\begin{align}
\sigma(HH^{\ast}) \ = \ \sigma(H^{\ast} H) \cup \zeta^2(X),
\end{align}
where $\zeta^2(X)$ denotes the set of eigenvalues not in $\sigma(H^{\ast} H)$, all of which are 0.
\item \emph{(III).} Suppose $\lambda^2 \in \sigma(H^{\ast} H)$ is associated to the following $\ell^2$-normalized eigenvectors:
\begin{align}
\mathbf{v}_{\ast} \ \leftrightsquigarrow \ H^{\ast} H, \quad \mathbf{v}_{\ast, +} \ \leftrightsquigarrow \ HH^{\ast}.
\end{align}
Then $\pm \lambda$ is associated to the following $\ell^2$-normalized eigenvector pair of $X$:
\begin{align}
\pm \lambda \ \leftrightsquigarrow \ \frac{1}{\sqrt{2}} \begin{pmatrix} \mathbf{v}_{\ast,+} \\ \pm \mathbf{v}_{\ast} \end{pmatrix}.
\end{align}
\item \emph{(IV).} Conversely, any eigenvalue pair $\pm \lambda \in \sigma^{1/2}(H^{\ast} H)$ is associated to the following $\ell^2$-normalized eigenvector pair of $X$:
\begin{align}
\pm \lambda \ \leftrightsquigarrow \ \frac{1}{\sqrt{2}} \begin{pmatrix} \mathbf{v}_{\ast, +} \\ \pm \mathbf{v}_{\ast} \end{pmatrix},
\end{align}
where $\mathbf{v}_{\ast, +}$ is an $\ell^2$-normalized eigenvector of $HH^{\ast}$ with eigenvalue $\lambda^2$ and $\mathbf{v}_{\ast}$ is an $\ell^2$-normalized eigenvector of $H^{\ast} H$ with eigenvalue $\lambda^2$. 
\item \emph{(V).} Suppose $\lambda = 0 \in \zeta^2(X)$ is associated to the $\ell^2$-normalized eigenvector $\mathbf{v}_{\lambda}$ of $HH^{\ast}$. Then for some $\lambda' = 0 \in \zeta(X)$, the corresponding $\ell^2$-normalized eigenvector is given by
\begin{align}
\lambda' \ \leftrightsquigarrow \ \begin{pmatrix} \mathbf{v}_{\lambda} \\ 0 \end{pmatrix}.
\end{align}
\item \emph{(VI).} Conversely, suppose $\lambda \in \zeta(X)$. Then $\lambda = 0$ is associated to the following $\ell^2$-normalized eigenvector of $X$:
\begin{align}
\lambda \ \leftrightsquigarrow \ \begin{pmatrix} \mathbf{v}_{\lambda} \\ 0 \end{pmatrix},
\end{align}
where $\mathbf{v}_{\lambda}$ is an $\ell^2$-normalized eigenvector of $HH^{\ast}$ with eigenvalue $\lambda' = 0$. 
\end{itemize}
\end{prop}
\begin{remark}
We briefly note that Proposition \ref{prop:spectralcorrespondence} applies to a much more general class of covariance matrices and their linearizations, as it does not refer to any underlying graph or graph structure.
\end{remark}
\begin{proof}
Statements (I) -- (II) are consequence of the SVD (singular value decomposition) of the matrix $H$ and dimension-counting. Statements (III) -- (VI) follow from a direct calculation and dimension-counting.
\end{proof}
\subsection{The Main Result}
We begin with notation for the Stieltjes transforms of the Marchenko-Pastur law and its linearization, respectively:
\begin{align}
m_{\infty}(z) \ &= \ \int_{\R} \ \frac{\varrho_{\infty}(x)}{x - z} \ \d x, \\
m(z) \ &= \ \int_{\R} \ \frac{\varrho(x)}{x - z} \ \d x.
\end{align}
Here, we take $z \in \C_+$ or $z \in \C_-$. We also define the following perturbed Stieltjes transforms to address the ensemble $\mathscr{X}_{\ast, +}$:
\begin{align}
m_{\infty,+}(z) \ &:= \ \gamma m_{\infty}(z) + \frac{\gamma - 1}{z} \ = \ \int_{\R} \ \frac{\gamma \varrho_{\infty}(x) + (\gamma - 1) \delta_0(x)}{x - z} \ \d x, \\
m_{+}(z) \ &:= \ \gamma m(z) + \frac{\gamma - 1}{z} \ = \ \int_{\R} \ \frac{\gamma \rho(x) + (\gamma - 1) \delta_0(x)}{x - z} \ \d x.
\end{align}
This describes the limiting spectral behavior. For the graphs themselves, we define the following Green's functions of the matrix ensembles $\mathscr{X}$, $\mathscr{X}_{\ast}$, and $\mathscr{X}_{\ast,+}$:
\begin{align}
G(z) \ &= \ (X - z)^{-1}, \quad X \in \mathscr{X}; \\
G_{\ast}(z) \ &= \ (X_{\ast} - z)^{-1}, \quad X_{\ast} \in \mathscr{X}_{\ast}; \\
G_{\ast,+}(z) \ &= \ (X_{\ast,+} - z)^{-1}, \quad X_{\ast,+} \in \mathscr{X}_{\ast, +}.
\end{align}
We also define the Stieltjes transforms of each covariance matrix ensemble $\mathscr{X}_{\ast}$ and $\mathscr{X}_{\ast,+}$, which may be realized as the Stieltjes transform of the corresponding empirical spectral distribution in spirit of Proposition \ref{prop:spectralcorrespondence}:
\begin{align}
s_{\ast}(z) \ = \ \frac{1}{N} \on{Tr} G_{\ast}(z), \quad s_{\ast,+}(z) \ = \ \frac{1}{M} \on{Tr} G_{\ast,+}(z).
\end{align}
Lastly, for the ensemble $\mathscr{X}$, we define instead the \emph{partial} Stieltjes transforms which average only over the diagonal terms of a specified color (black or whtie):
\begin{align}
s_b(z) \ = \ \frac{1}{M} \sum_{i = 1}^M \ G_{ii}(z), \quad s_w(z) \ = \ \frac{1}{N} \sum_{k = 1}^N \ G_{kk}(z).
\end{align}
We now introduce domains in the complex plane on which we study the Green's functions of each matrix ensemble. These domains are engineered to avoid the singularities in the Green's functions near the origin, and in the case of the linearized Marchenko-Pastur law, the edge of the support. To this end, we establish notation for the following subsets of the complex plane for any fixed $\e > 0$:
\begin{align}
U_{\e, \pm} \ &:= \ \left\{ z = E + i \eta: \ |E| > \e, \ \eta > 0 \right\}, \\
U_{\e} \ &:= \ U_{\e,+} \cup U_{\e, -}.
\end{align}
We will also need to define the following control parameters:
\begin{align}
D \ &:= \ d_b \wedge \frac{N^2}{d_b^3}, \\
\Phi(z) \ &:= \ \frac{1}{\sqrt{N \eta}} + \frac{1}{\sqrt{D}}, \\
F_z(r) \ = \ F(r) \ &:= \ \left[ \left(1 + \frac{1}{\sqrt{(\lambda_+ - z)(z - \lambda_-)}} \right) r \right] \ \wedge \ \sqrt{r}.
\end{align}
We now present the main result of this paper.
\begin{theorem} \label{theorem:locallaw}
Suppose $\xi = \xi_N$ is a parameter chosen such that the following growth conditions on $D$ and $\eta$ hold:
\begin{align}
\xi \log \xi \ \gg \ \log^2 N, \quad |\eta| \ \gg \ \frac{\xi^2}{N}, \quad D \ \gg \ \eta^2. \label{eq:growthconditionslocallaw}
\end{align}
Then for any fixed $\e > 0$, we have the following estimates with probability at least $1 - e^{-\xi \log \xi}$, uniformly over all $z = E + i \eta \in U_{\e}$ with $\eta$ satisfying the growth condition in \eqref{eq:growthconditionslocallaw}:
\begin{align}
\max_{i} \left| [G_{\ast}(z)]_{ii} - m_{\infty}(z) \right| \ = \ O \left( F_z(\xi \Phi) \right), \ \ \ \max_{i \neq j} \left| [G_{\ast}(z)]_{ij} \right| \ = \ O\left( \frac{\xi \Phi(z^2)}{z} \right). \label{eq:GFlocallawsmall}
\end{align}
Similarly, for any fixed $\e > 0$, we have the following estimates with probability at least $1 - e^{-\xi \log \xi}$, uniformly over all parameters $z = E + i \eta \in U_{\e}$ with $\eta$ satisfying the growth condition in \eqref{eq:growthconditionslocallaw}:
\begin{align}
\max_{i} \left| [G_{\ast,+}(z)]_{ii} - m_{\infty,+}(z) \right| \ = \ O\left( F_z( \xi \Phi) \right), \ \ \ \max_{i \neq j} \left| [G_{\ast,+}(z)]_{ij} \right| \ = \ O \left( \frac{\xi \Phi(z^2)}{z} \right). \label{eq:GFlocallawlarge}
\end{align}
Conditioning on the estimates \eqref{eq:GFlocallawsmall} and \eqref{eq:GFlocallawlarge}, respectively, uniformly over $z = E + i \eta \in U_{\e}$ with $\eta$ satisfying the growth condition in \eqref{eq:growthconditionslocallaw}, we have
\begin{align}
\left| s_{\ast}(z) - m_{\infty}(z) \right| \ = \ O \left( F_z(\xi \Phi) \right), \quad \left| s_{\ast,+}(z) - m_{\infty}(z) \right| \ = \ O \left( F_z(\xi \Phi) \right). \label{eq:STlocallawcov}
\end{align}
Conditioning on the estimates \eqref{eq:GFlocallawsmall}, uniformly over all $z = E + i \eta$ with $\eta$ satisfying the growth condition in \eqref{eq:growthconditionslocallaw}, we have
\begin{align}
\max_{k > M} \left| [G(z)]_{kk} - m(z) \right| \ &= \ O \left( z F_{z^2}(\xi \Phi(z^2)) \right), \label{eq:GFlocallawlinearsmalldiagonal} \\
\max_{M < k < \ell} \ \left| [G(z)]_{k \ell} \right| \ &= \ O \left(\xi \Phi \right).\label{eq:GFlocallawlinearsmalloffdiagonal}
\end{align}
Moreover, conditioning on the estimates \eqref{eq:GFlocallawlarge}, for any fixed $\e > 0$, uniformly over all $z = E + i \eta \in U_{\e}$ with $\eta$ satisfying the growth condition in \eqref{eq:growthconditionslocallaw}, we have
\begin{align}
\max_{i \leq M} \left| [G(z)]_{ii} - m_{+}(z) \right| \ &= \ O \left( z F_{z^2}(\xi \Phi(z^2)) \right),  \label{eq:GFlocallawlinearlargediagonal} \\
\max_{i < j \leq M} \ \left| [G(z)]_{ij} \right| \ &= \ O(\xi \Phi).\label{eq:GFlocallawlinearlargeoffdiagonal}
\end{align}
Conditioning on \eqref{eq:GFlocallawsmall} and \eqref{eq:GFlocallawlarge}, we have the following estimates uniformly over all $z = E + i \eta$ with $\eta$ satisfying the growth condition in \eqref{eq:growthconditionslocallaw}:
\begin{align}
\left| s_b(z) - m_{+}(z) \right| \ &= \ O \left( z F_{z^2}(\xi \Phi(z^2)) \right), \label{eq:partialSTblack} \\
\left| s_w(z) - m(z) \right| \ &= \ O \left( z F_{z^2}(\xi \Phi(z^2)) \right). \label{eq:partialSTwhite}
\end{align}
Lastly, the estimates \eqref{eq:GFlocallawsmall} and \eqref{eq:STlocallawcov} hold without the condition $|E| > \e$ if $\alpha > 1$. The estimates \eqref{eq:GFlocallawlinearlargediagonal} and \eqref{eq:GFlocallawlinearlargeoffdiagonal} hold without the condition $|E| > \e$ if $\alpha = 1$.
\end{theorem}
\begin{remark}
We briefly remark on the repulsion assumption $|E| > \e$ in Theorem \ref{theorem:locallaw}. The removal of this assumption discussed at the end of the statement of Theorem \ref{theorem:locallaw} is a direct consequence of studying the dependence of the singularities of the Green's functions and Stieltjes transforms at the origin with respect to the structural parameter $\alpha$. For example, the presence of a singularity of $m_{\infty}$ at the origin occurs exactly when $\alpha = 1$. Moreover, the singularities in the Stieltjes transforms of matrices and the singularities of $m_{\infty,+}$ at the origin cancel each other out, allowing for a regularization at the origin.
\end{remark}
\begin{remark}
We last remark that if $\alpha = 1$, the covariance matrices $X_{\ast}$ and $X_{\ast,+}$ are equal in law. This comes from symmetry of the bipartite graph between the two vertex sets $V_b$ and $V_w$, i.e. the graph statistics are unchanged upon relabeling the graph. This allows us to remove the assumption $|E| > \e > 0$ for certain estimates in Theorem \ref{theorem:locallaw} in the regime $\alpha = 1$.
\end{remark}
We now discuss important consequences of Theorem \ref{theorem:locallaw}, the first of which is the following result on eigenvector delocalization, i.e. an estimate on the $\ell^{\infty}$-norm of an eigenvector in terms of its $\ell^2$-norm. The proof of this delocalization result will be delegated to a later section after we study in more detail the spectral data of covariance matrices and their linearizations.
\begin{corollary} \label{corollary:delocalization}
\emph{(Eigenvector Delocalization).}

Assume the setting of Theorem \ref{theorem:locallaw}, and suppose $\mathbf{u}$ is an eigenvector of $X_{\ast}$ with eigenvalue $\lambda$. Then with probability at least $1 - e^{-\xi \log \xi}$, we have
\begin{align}
\| \mathbf{u} \|_{\ell^{\infty}} \ = \ O \left( \frac{\xi}{\sqrt{N}} \| \mathbf{u} \|_{\ell^2} \right).
\end{align}
\end{corollary}
We briefly remark that the eigenvector delocalization fails for the larger covariance matrix $X_{\ast,+}$. 
\begin{proof}
First, we note in the case $\mathbf{u} \in \on{Span}(\mathbf{e}_b)$, the result is true trivially. Moreover, by Proposition \ref{prop:spectralcorrespondence}, it suffices to prove the claim for eigenvectors of the linearization $X$, replacing the $\ell^{\infty}$-norm by a supremum over indices $k > M$.

We now take for granted $|zm_{\infty}(z^2)| = O(1)$ uniformly for $z = E + i \eta \in \C_+$; this follows from an elementary analysis of the Stieltjes transform discussed in the appendix of this paper. This allows us to obtain the following string of inequalities with probability at least $1 - e^{-\xi \log \xi}$ and any index $k > M$:
\begin{align}
\left| \mathbf{u}(k) \right|^2 \ &\leq \ \sum_{\mathbf{v}_{\beta} \neq \mathbf{u}} \ \frac{\eta^2 \left| \mathbf{v}_\beta(k) \right|^2}{(\lambda_\beta - \lambda)^2 + \eta^2} \\ 
&= \ \eta \on{Im} [G(\lambda + i \eta)]_{kk} \\ 
&\leq \ \eta \left| zm_{\infty}(z^2) \right| + O(\eta \sqrt{\xi \Phi}) \\
&\leq \ 2 \eta,
\end{align}
where we used the local law for the linearization $X$ to estimate the second line. This completes the derivation of the eigenvector delocalization.
\end{proof}
We conclude this preliminary discussion concerning consequences of Theorem \ref{theorem:locallaw} with the following weak rigidity estimates. We briefly remark that it relies heavily upon the Helffer-Sjostrand formula and functional calculus, and beyond these tools, the local law in Theorem \ref{theorem:locallaw}. To state the result, we first introduce the following definition.
\begin{definition}
For each $i \in [[1, N]]$, we define the $i$-th \emph{classical location}, denoted $\gamma_i$, by the following quantile formula:
\begin{align}
\frac{i}{N} \ = \ \int_{-\infty}^{\gamma_i} \ \varrho_{\infty}(E) \ \d E,
\end{align}
where we recall $\varrho_{\infty}$ denotes the density function of the Marchenko-Pastur law.
\end{definition}
The following consequence of Theorem \ref{theorem:locallaw} will compare the classical location $\gamma_i$ to the $i$-th eigenvalue $\lambda_i$ of the covariance matrix $X_{\ast}$, where the ordering on the eigenvalues is the increasing order.
\begin{corollary}
For any fixed $\kappa > 0$ and index $i \in [[\kappa N, (1-\kappa)N]]$, we have, with probability at least $1 - e^{-\xi \log \xi}$,
\begin{align}
\left| \lambda_i - \gamma_i \right| \ = \ O \left( \frac{\xi^2}{D^{1/4}} \right).
\end{align}
\end{corollary}
For details of the proof, we refer to Section 5 in \cite{BHKY} and Section 7 in \cite{LY}.

We now give an outline for the derivation of the local law. The proof will roughly consist of the following three steps:
\begin{itemize}
\item (I). The first step will be to adapt the methods in \cite{BKY} to define and study a method of resampling biregular graphs in $\Omega$. The resampling will be generated by local operations on a given graph known as \emph{switchings}, which we will define more precisely in a later section. The local nature of the resampling method will help us derive equations exploiting the probabilistic stability of the Green's function under these switchings.
\item (II). The second step will be to study the Green's functions of the three matrix ensemble simultaneously. This includes both a preliminary analysis and a further analysis using the switching dynamics established in the previous step. In particular, we derive an approximate self-consistent equation for the diagonal entries of the Green's function and study its stability properties. As in \cite{BKY}, this will help us compare the diagonal of the Green's function to the associated Stieltjes transform. The equation in \cite{BKY}, however, contains a constant leading-order coefficient whereas for covariance matrices the leading-order coefficient is nonconstant. We adapt the methods suitably to handle this nonlinearity.
\end{itemize}
%
%
%
\section{Switchings on Bipartite Graphs}
We begin by introducing notation necessary to define switchings on biregular graphs. Switchings will be local operations on the biregular graphs, so we will establish notation for vertices and edges containing said vertices as follows.
\begin{notation}
A generic vertex in $V_b$ (resp. $V_w$) will be denoted by $v_b$ (resp. $v_w$). 

For a fixed graph $E \in \Omega$, we will denote the edges in $E$ containing $v_b$ by $\{ e_{b, \mu} \}_{\mu = 1}^{d_b}$. Moreover, for a fixed edge $e_{b, \mu}$ containing $v_b$, we will denote the neighboring vertex by $v_{b,\mu}$, so that $e_{b, \mu} = v_b v_{b, \mu}$. 

Similarly, the edges in $E$ containing $v_w$ will be denoted by $\{ e_{w, \nu} \}_{\nu = 1}^{d_w}$. For a fixed edge $e_{w,\nu}$ containing $v_w$, we will denote the neighboring vertex by $v_{w, \nu}$.

For a fixed vertex $v_b \in V_b$, we establish the notation for the set of edges not containing $v_b$:
\begin{align}
U_{v_b} \ := \ \left\{ \on{edges} \ e \in E: \ v_b \not\in e \right\}.
\end{align}
Similarly for a fixed vertex $v_w \in V_w$, we define the following set of edges not containing $v_w$:
\begin{align}
U_{v_w} \ := \ \left\{ \on{edges} \ e \in E: \ v_w \not\in e \right\}.
\end{align}
\end{notation}
We may now begin to define a switching on a generic graph $E \in \Omega$. To this end, we fix a black vertex $v_b \in V_b$ and an edge $e_{v, \mu}$ for some $\mu \in [[1, d_b]]$. We define the following space of subgraphs of $E$:
\begin{align}
\mathbf{S}_{v_b, \mu, E} \ := \ \left\{ S \subset E: \ S \ = \ \{ e_{b, \mu}, p_{b, \mu}, q_{b,\mu} \}, \ p_{b,\mu} \neq q_{b, \mu} \in U_{v_b} \right\}.
\end{align}
In words, the set $\mathbf{S}_{v_b, \mu, E}$ is the set of graphs consisting of the edges $e_{b,\mu}$ and any two distinct edges $p_{b,\mu}$ and $q_{b,\mu}$, neither of which contains the vertex $v_b$. Similarly, we may define for a fixed white vertex $v_w \in V_w$ and edge $e_{w, \nu}$, for some $\nu \in [[1, d_w]]$, the same set of graphs:
\begin{align}
\mathbf{S}_{v_w, \mu, E} \ := \ \left\{ S \subset E: \ S \ = \ \{ e_{w, \nu}, p_{w, \nu}, q_{w, \nu} \}: \ p_{w, \nu} \neq q_{w,\nu} \in U_{v_w} \right\}.
\end{align}
%
\begin{notation}
A generic graph in $\mathbf{S}_{v_b, \mu, E}$ will be denoted by $S_{b, \mu}$. A generic graph in $\mathbf{S}_{v_w, \nu, E}$ will be denoted by $S_{w, \nu}$. 
\end{notation}
The set $\mathbf{S}_{v_b, \mu, E}$ contains the edge-local data along which switchings on graphs will be defined. To make this precise, we need to introduce the following indicator functions. First, we define the following configuration vectors for fixed vertices $v_b \in V_b$ and $v_w \in V_w$:
\begin{align}
\mathbf{S}_{v_b} \ := \ \left( S_{b, \mu} \right)_{\mu = 1}^{d_b}, \quad \mathbf{S}_{v_w} \ := \ \left( S_{w, \nu} \right)_{\nu = 1}^{d_w}.
\end{align}
With this notation, we define the following indicator functions that detect graph properties in $S_{b, \mu}$ and $S_{w, \nu}$. 
\begin{align}
I(S_{b, \mu}) \ &= \ \mathbf{1} \left( |[S_{b, \mu}]| = 6 \right), \label{eq:I} \\
J(\mathbf{S}_{v_b}, \mu) \ &= \ \prod_{\mu' \neq \mu} \ \mathbf{1} \left( [S_{b, \mu}] \cap [S_{b, \mu'}] = \{ v_b \} \right), \label{eq:J} \\
W(\mathbf{S}_{v_b}) \ &= \ \left\{ \mu: \ I(S_{b, \mu}) J(\mathbf{S}_{v_b}, \mu) = 1 \right\}. \label{eq:W}
\end{align}
For white vertices $v_w \in V_w$, the functions $I, J$ and $W$ retain the same definition upon replacing $b$ with $w$ and $\mu$ with $\nu$.

We now define the augmented probability spaces $\wt{\Omega}$ which will make the switchings systematic from the perspective of Markovian dynamics. For a fixed black vertex $v_b \in V_b$ and a fixed white vertex $v_w \in V_w$, we define the following augmented space:
\begin{align}
\wt{\Omega} \ &= \left\{ \left( E, \mathbf{S}_{v_b}, \mathbf{S}_{v_w} \right), \quad E \in \Omega, \ \mathbf{S}_{v_b} \in \prod_{\mu = 1}^{d_b} \mathbf{S}_{v_b, \mu, E}, \ \mathbf{S}_{v_w} \in \prod_{\nu = 1}^{d_w} \mathbf{S}_{v_w, \nu, E} \right\}.
\end{align}
We now precisely define switchings by defining dynamics on $\wt{\Omega}$. To this end we define switchings on configuration vectors $\mathbf{S}_{v_b}$ and $\mathbf{S}_{v_w}$; we first focus on the configuration vectors for black vertices. 

Fix a label $\mu$ and consider a component $S_{b, \mu}$ of a uniformly sampled configuration vector $\mathbf{S}_{v_b}$. Precisely, the components of $\mathbf{S}_{v_b}$ are sampled jointly uniformly and independently from $\mathbf{S}_{v_b, \mu, E}$, where $E \in \Omega$ is uniform over all $\mu$ and sampled uniformly. We now define the following map:
\begin{align}
T_b: \prod_{\mu = 1}^{d_b} \ \mathbf{S}_{v_b, \mu, E} \ \longrightarrow \ \prod_{\mu = 1}^{d_b} \ \mathbf{S}_{v_b, \mu, E'}
\end{align}
where $E' \in \Omega$ is possibly different from $E$. The map is given as follows: for any $\mu$, we define the map $T_{b,\mu}$
\begin{align}
T_{b,\mu}(S_{b, \mu}) \ = \ \begin{cases}
	S_{b, \mu} & \mu \not\in W(\mathbf{S}_{v_b}) \\
	(S_{b, \mu}, s_{b, \mu}) & \mu \in W(\mathbf{S}_{v_b})
\end{cases}.
\end{align}
We define the graph $(S_{b, \mu}, s_{b,\mu})$ as follows; this is where we now introduce randomness into the dynamics $T_b$. Suppose $\mu \in W(\mathbf{S}_{v_b})$, in which case $S_{b,\mu}$ is 1-regular and bipartite with respect to the vertex sets $([S_{v_b}], V_1, V_2)$. Consider the set of 1-regular bipartite graphs with respect to the vertex set $([S_{b,\mu}], V_1, V_2)$. In words, this is the set of 1-regular graphs on $[S_{b, \mu}]$ such that, upon replacing $S_{b,\mu}$ with any such graph, the global graph $E$ remains biregular. We now define $(S_{b, \mu}, s_{b,\mu})$ to be drawn from this set uniformly at random conditioning on the event $(S_{b, \mu}, s_{b, \mu}) \neq S_{b,\mu}$. Lastly, we define the following global dynamics:
\begin{align}
T_b \ = \ \prod_{\mu = 1}^{d_b} \ T_{b,\mu},
\end{align}
where the product is taken as composition. We note this product is independent of the order of composition; this is a consequence of the definition of the functions $I, J$ and $W$. For white vertices $v_w \in V_w$, we define the map $T_w$ by replacing all black indices $b$ and white indices $w$.

We note that the maps $T_b$ and $T_w$ define maps on $\wt{\Omega}$, because we are allowed to change the underlying graph $E$ when varying over the space $\wt{\Omega}$; this is the utility of the almost-product representation of $\wt{\Omega}$. This allows us to finally define switchings of a biregular graph.
\begin{definition}
For a fixed black vertex $v_b \in V_b$ and a fixed label $\mu \in [[1, d_b]]$, the \emph{local switching} at $v_b$ along $\mu$ is the map $T_{b,\mu}$. The \emph{global switching} is the map $T_b$.

Similarly, for a fixed white vertex $v_w \in V_w$ and a fixed label $\nu \in [[1, d_w]]$, the \emph{local switching} at $v_w$ along $\nu$ is the map $T_{w, \nu}$. The \emph{global switching} at $v_w$ is the map $T_w$.
\end{definition}
\begin{remark}
We note our construction, technically, implies the mappings $T_{b, \mu}$ and $T_{w, \nu}$ are random mappings on the augmented space $\wt{\Omega}$.  Via this construction, we obtain a probability measure $\wt{\Omega}$ induced by the uniform measure and a uniform sampling of switchings. To obtain an honest mapping on the original space $\Omega$, we may instead construct deterministic mappings by averaging over the random switchings. For precise details, we cite \cite{BKY}.
\end{remark}
\subsection{Switchings on Adjacency Matrices}
We now aim to translate the combinatorics of graph switchings into analysis of adjacency matrices. Suppose $E \in \Omega$ is a biregular graph with adjacency matrix $A$. We will fix the following notation.
\begin{notation}
For an edge $e = ij$ on the vertex set $\mathscr{V}$, we let $\Delta_{ij}$ denote the adjacency matrix of the graph on $\mathscr{V}$ consisting only of the edge $e$. In particular, $\Delta_{ij}$ is the matrix whose entries are given by
\begin{align}
(\Delta_{ij})_{k\ell} \ = \ \delta_{ik} \delta_{j \ell} + \delta_{i \ell} \delta_{jk}.
\end{align}
\end{notation}
In the context of switchings on biregular graphs, the matrices $\Delta_{ij}$ are perturbations of adjacency matrices. This is made precise in the following definition.
\begin{definition}
Fix a black vertex $v_b \in V_b$ and a label $\mu \in [[1, d_b]]$. A \emph{local} switching of $A$, denoted $T_{b,\mu}$, at $v_b$ along $\mu$ is given by the following formula:
\begin{align}
T_{b,\mu}(A) \ = \ A - S_{b,\mu} + (S_{b, \mu}, s_{b,\mu}), \label{eq:switchmatrix}
\end{align}
where $S_{b,\mu}$ is a component of a random, uniformly sampled configuration vector $\mathbf{S}_{v_b}$. A \emph{global} switching of $A$, denoted $T_b$, is the composition of the random mappings $T_{b,\mu}$.

Similarly, we may define local switchings and global switchings of adjacency matrices for white vertices by replacing the black subscript $b$ with the white subscript $w$, and replacing the label $\mu$ with $\nu$.
\end{definition}
Clearly, a local or global switching of an adjacency matrix is the adjacency matrix corresponding to a local or global switching of the underlying graph. To realize the matrices $\Delta_{ij}$ as perturbations, we will rewrite the formula defining $T_{b,\mu}$ as follows. As usual, we carry out the discussion for black vertices $v_b \in V_b$, though the details for white vertices $v_b$ follow analogously. 

First, we recall the following notation for a component $S_{b,\mu} \in \mathbf{S}_{v_b, \mu, E}$ of a configuration vector $\mathbf{S}_{v_b}$:
\begin{align}
S_{b,\mu} \ := \ \left\{ e_{b,\mu}, p_{b,\mu}, q_{b,\mu} \right\},
\end{align}
subject to the constraint that $S_{b,\mu}$ contains three distinct edges. 
\begin{notation}
We will denote the vertices of $p_{b,\mu}$ by $a_{b,p,\mu} \in V_b$ and $a_{w,p,\mu} \in V_w$. Similarly, we will denote the vertices of $q_{b,\mu}$ by $a_{b,q,\mu} \in V_b$ and $a_{w,q,\mu} \in V_w$.
\end{notation}
With this notation, we may rewrite the random mapping $T_{b,\mu}$ as follows:
\begin{align}
T_{b,\mu}(A) \ = \ A - \left( \Delta_{v_b, v_{b,\mu}} +  \Delta_{a_{b,p,\mu} a_{w,p,\mu}} + \Delta_{a_{b,q,\mu} a_{w,q,\mu}} \right) + \left( \Delta_{v_b, x} + \Delta_{a_{b,p,\mu}, y} + \Delta_{a_{b,q,\mu}, z} \right), \label{eq:perturbswitchmatrix}
\end{align}
where we recall $e_{b,\mu} = v_b v_{b,\mu}$. Here, the variables $x, y, z$ are three distinct vertices sampled  from the set of white vertices $\{ v_{b,\mu}, a_{w,p,\mu}, a_{w,q,\mu}\}$ conditioning on the following constraint on ordered triples:
\begin{align}
(x, y, z) \ \neq \ (v_{b,\mu}, a_{w,p,\mu}, a_{w,q,\mu}).
\end{align}
%
\subsection{Probability Estimates on Vertices}
In this discussion, we obtain estimates on the distribution of graph vertices \emph{after} performing switchings. The main estimates here show that the vertices are approximately uniformly distributed, which we make precise in the following definition.
\begin{definition} \label{definition:approxuniform}
Suppose $S$ is a finite set and $X$ is an $S$-valued random variable. We say (the distribution of) $X$ is \emph{approximately uniform} if the following bound on total variation holds:
\begin{align}
\sum_{s \in S} \ \left| \mathbb{P} \left( X = s \right) - \frac{1}{|S|} \right| \ \leq \ O \left( \frac{1}{\sqrt{d_w D}} \right).
\end{align}
\end{definition}
We now introduce the following $\sigma$-algebras on $\wt{\Omega}$. These $\sigma$-algebras will allow us to focus on edge-local features of graphs $E \in \Omega$ upon conditioning on the global data $E$.
\begin{definition}
For a fixed label $\mu \in [[0, d_b]]$, we define the following $\sigma$-algebras:
\begin{align}
\mathscr{F}_{\mu} \ &:= \ \sigma \left( E, (S_{b,1}, s_{b,1}), \ldots, (S_{b,\mu}, s_{b,\mu}) \right), \\
\mathscr{G}_{\mu} \ &:= \ \sigma \left( E \left( S_{b, \mu'}, s_{b, \mu'} \right)_{\mu' \neq \mu} \right).
\end{align}
We similarly define the $\sigma$-algebras $\mathscr{F}_{\nu}$ and $\mathscr{G}_{\nu}$ for $\nu \in [[1, d_w]]$ for white vertices.
\end{definition}
In particular, conditioning on $\mathscr{F}_0$ corresponds to conditioning on the graph $E$ only. The last piece of probabilistic data we introduce is the following notation, which will allow us to compare i.i.d. switchings on biregular graphs.
\begin{notation}
Suppose $X$ is a random variable on the graph data $E, \{ (S_{b, \mu}, s_{b,\mu}) \}_{\mu},  \{(S_{w, \nu}, s_{w, \nu}) \}_{\nu}$. Then $\wt{X}$ denotes a random variable on the variables $\wt{E}, \{ (\wt{S}_{b, \mu}, \wt{s}_{b,\mu})_{\mu}, \{ (\wt{S}_{w, \nu}, \wt{s}_{w,\nu}) \}$, where the tildes on the graph data denote i.i.d. resamplings. 
\end{notation}
\begin{notation}
For notational simplicity, by $p_{\mu}, q_{\mu}$ or $p_{\nu}, q_{\nu}$, we will refer to either $p_{b, \mu}, q_{b, \mu}$ or $p_{w, \nu}, q_{w, \nu}$, respectively, whenever the discussion applies to both situations.
\end{notation}
We now focus on obtaining an estimate on the distribution of the pair of edges $(p_{\mu}, q_{\mu})$, and similarly for $(p_{\nu}, q_{\nu})$. As with all results concerning switchings from here on, details of proofs resemble those of Section 6 in \cite{BKY}, so we omit details whenever redundant.
\begin{lemma} \label{lemma:approxuniformedges}
Conditioned on $\mathscr{G}_{\mu}$, the pair $(p_\mu, q_\mu)$ is approximately uniform, i.e., for any bounded symmetric function $F$, we have
\begin{align}
\E_{\mathscr{G}_\mu} F(p_\mu, q_\mu) \ = \ \frac{1}{(N d_w)^2} \sum_{p,q \in E} \ F(p,q) \ + \ O \left( \frac{1}{N} \| F \|_{\infty} \right). \label{eq:approxuniformjointpair}
\end{align}
Similarly, for any bounded function $F$, we have
\begin{align}
\E_{\mathscr{G}_{\mu}, q_\mu} F(p_\mu) \ = \ \frac{1}{N d_w} \sum_{p \in E} \ F(p) \ + \ O \left( \frac{1}{N} \| F \|_{\infty} \right). \label{eq:approxuniformmarginal}
\end{align}
\end{lemma}
\begin{proof}
Assume we resample about $v_b \in V_b$; the case for $v_w \in V_w$ follows analogously. By definition, we have
\begin{align}
\E_{\mathscr{G}_{\mu}} F(p_{\mu}, q_\mu) \ = \ \frac{1}{(Nd_w - d_w)(Nd_w - d_w - 1)} \sum_{p \in E_{v_b}} \ F(p),
\end{align}
where $E_{v_b}$ is the set of edges in $E$ that are not incident to $v_b$. Then, \eqref{eq:approxuniformjointpair} follows from the following estimate
\begin{align}
\frac{1}{(Nd_w - d_w)(Nd_w - d_w - 1)} \ = \ \frac{1}{(N d_w)^2} \ + \ O \left( \frac{1}{N^3 d_w^2} \right)
\end{align}
as well as the estimate $| E_{v_b} | \leq (N d_w)^2$, and lastly the estimate $|E_{v_b}^C| \leq Nd_w$. This last upper bound follows combinatorially; for details, see the proof of Lemma 6.2 in \cite{BKY}. The estimate \eqref{eq:approxuniformmarginal} follows from a similar argument.
\end{proof}
Because an edge is uniquely determined by its vertices in the graph, we automatically deduce from Lemma \ref{lemma:approxuniformedges} the following approximately uniform estimate for resampled vertices as well.
\begin{corollary} \label{corollary:switchingverticesapproxuniform}
Conditioned on $\mathscr{G}_\mu$, the pair $(p_{\mu}(b), q_{\mu}(b))$ (resp. $(p_{\mu}(w), q_{\mu}(w))$) is approximately uniform. 

Similarly, conditioned on $\mathscr{G}_\mu$ and $q_{\mu}(b)$ (resp. $q_{\mu}(w)$), the random variable $p_{\mu}(b)$ (resp. $p_{\mu}(w)$) is approximately uniform.
\end{corollary}
\begin{proof}
This follows immediately upon applying Lemma \ref{lemma:approxuniformedges} to the function $F(p_{\mu}, q_{\mu}) = f(p_{\mu}(b), q_\mu(b))$.
\end{proof}
To fully exploit the resampling dynamics, we need a lower bound on the probability that a local switching $S_{b, \mu}, s_{b, \mu}$ around a vertex $v_b \in V_b$ does not leave the graph fixed. In particular, we need an estimate for the probability of the event $\mu \in W(\mathbf{S}_{v_b})$ where here $\mu$ is fixed and the set $W$ is viewed as random. As discussed in \cite{BKY}, to provide an estimate, the naive approach to estimating this probability conditioning on $\mathscr{G}_\mu$ fails in an exceptional set. Precisely, suppose the $\mu$-th neighbor $v_{b, \mu}$ of $v_b$ lives in $S_{b, \mu'}$ for some $\mu' \neq \mu$. In this case, almost surely, we have $[S_{b, \mu}] \cap [S_{b, \mu'}]$ is nontrivial. It turns out this is the only obstruction, so we aim to show that $v_{b, \mu} \in [S_{b, \mu'}]$ occurs with low probability for any $\mu' \neq \mu$. 

Formally, we define the following indicator random variable which detects this exceptional set:
\begin{align}
h(\mathbf{S}_{v_b}, \mu) \ = \ \prod_{\mu' \neq \mu} \ \mathbf{1} \left( v_{b, \mu} \in S_{b, \mu'} \right).
\end{align}
Thus, the estimates we need are given in the following result.
\begin{lemma} \label{lemma:distributionneighbors}
For any neighbor index $\mu$, we have
\begin{align}
\mathbb{P}_{\mathscr{G}_\mu} \left[ I(S_{b, \mu}) J(\mathbf{S}_{v_b},\mu) \ = \ h(\mathbf{S}_{v_b}, \mu) \right] \ \geq \ 1 - O \left( \frac{d_b}{N} \right). \label{eq:distributionneighborsestimate1}
\end{align}
Moreover, we have
\begin{align}
\mathbb{P}_{\mathscr{F}_0} \left[ h(\mathbf{S}_{v_b}, \mu) = 1 \right] \ \geq \ 1 - O \left( \frac{d_b}{N} \right). \label{eq:distributionneighborsestimate2}
\end{align}
\end{lemma}
\begin{proof}
We first note that \eqref{eq:distributionneighborsestimate1} follows immediately conditioning on $h = 0$. In particular, the first lower bound \eqref{eq:distributionneighborsestimate1} follows from a combinatorial analysis of the underlying graph using the following union bound:
\begin{align}
\mathbb{P}_{\mathscr{G}_\mu, h = 1} \left[ I(S_{b, \mu}) J(\mathbf{S}_{v_b}, \mu) = 0 \right] \ \leq \ &\mathbb{P}_{\mathscr{G}_\mu} \left[ I(S_{b, \mu}) = 0 \right]  \nonumber \\
&\quad \ + \ \mathbb{P}_{\mathscr{G}_\mu, h = 1} \left[ J(\mathbf{S}_{v_b}, \mu) = 0 \right].
\end{align}
Similarly, \eqref{eq:distributionneighborsestimate2} follows from the union bound
\begin{align}
\mathbb{P}_{\mathscr{F}_0} \left[ h(\mathbf{S}_{v_b}, \mu) = 0 \right] \ \leq \ \sum_{\mu' \neq \mu} \ \mathbb{P}_{\mathscr{F}_0} \left[ v_{b, \mu} \in [S_{\mu'}] \right].
\end{align}
For details, we refer back to \cite{BKY}.
\end{proof}
We conclude this section with an estimate that compares independent resamplings. Recall that $\wt{W}, W$ are i.i.d. copies of the random variable $W(\mathbf{S}_{v_b})$. The following result bounds the fluctuation in $W(\mathbf{S}_{v_b})$ from independent resamplings. 
\begin{lemma} \label{lemma:howtocomputeswitchperturbations}
Almost surely, we know 
\begin{align}
\# \left( W \Delta \wt{W} \right) \ = \ O(1),
\end{align}
where the implied constant is independent of $N$. Moreover, we also have
\begin{align}
\mathbb{P}_{\mathscr{G}_\mu} \left[ W \Delta \wt{W} \neq \emptyset \right] \ \leq \ O \left( \frac{d_b}{N} \right).
\end{align}
\end{lemma}
The proof follows the argument concerning Lemma 6.3 in \cite{BKY} almost identically, so we omit it. We now present the final estimate on adjacency matrices comparing switched matrices upon i.i.d. switchings in the sense of matrix perturbations. This will allow us to perform and control resamplings of biregular graphs, in particular using the resolvent perturbation identity.
\begin{lemma} \label{lemma:switchperturbmatrix}
Under the setting of the resampling dynamics, we have
\begin{align}
\wt{A} - A \ = \ T_{b, \mu}(A) - \wt{T}_{b,\mu}(\wt{A}) \label{eq:highprobabilityperturbationestimate}
\end{align}
with probability at least $1 - O(d_b/N)$. Almost surely, we have
\begin{align}
\wt{A} - A \ = \ \sum_{x, y = 1}^{O(1)} \ \Delta_{xy}, \label{eq:almostsureperturbationestimate}
\end{align}
such that either, conditioning on $\mathscr{G}_\mu$, the random indices $x,y$ are approximately uniform in the corresponding set $V_b$ or $V_w$ or, conditioning on $\mathscr{G}_{\mu}, p_{\mu}, \wt{p}_{\mu}$, at least one of the random indices $x, y$ is approximately uniform in the appropriate vertex set. 

Lastly, the statement remains true upon switching instead at a white vertex $v_w$ along an edge label $\nu$.
\end{lemma}
\begin{proof}
The result follows from unfolding Lemma \ref{lemma:howtocomputeswitchperturbations} and the following deterministic identity:
\begin{align}
\wt{A} - A \ = \ &\mathbf{1}_{\mu \in \wt{W}} \left[ \wt{T}_{b, \mu}(E) - E \right] \ - \ \mathbf{1}_{\mu \in W} \left[ T_{b, \mu}(A) - A \right] \nonumber \\
&\quad \ + \ \sum_{\mu' \in \wt{W} \Delta W} \ \pm \left[ T_{b,\mu}(A) - A \right],
\end{align}
where the sign corresponds to which of the random sets $W$ or $\wt{W}$ contains the indexing label $\mu'$.
\end{proof}

%
%
%
\section{Green's Function Analysis}
\subsection{Preliminary Resolvent Theory}
Here we record the following fundamental identities for studying the Green's functions of adjacency matrices in the ensemble $\mathscr{X}$. These identities are standard and follow from standard linear algebra. First, because these identities hold for Green's functions of any real symmetric matrix, we fix the following notation.
\begin{notation}
Suppose $F$ is a function of the Green's function or matrix entries of matrices belonging to any one of the matrix ensembles $X$, $X_{\ast}$, or $X_{\ast,+}$. Then we establish the notation $F_{\star}$ to be the function obtained when restricted to the matrix ensemble $X_{\star}$, where we take $\star$ to be blank or $\star = \ast$ or $\star = \ast, +$.  
\end{notation}
\begin{lemma} \label{lemma:resolventidentity}
\emph{(Resolvent Identity)}

Suppose $A$ and $B$ are invertible matrices. Then we have
\begin{align}
A^{-1} - B^{-1} \ = \ A^{-1}(B - A)B^{-1}. \label{eq:generalresolventidentity}
\end{align}
In particular, if $H$ and $\wt{H}$ denote real symmetric or complex Hermitian matrices with Green's functions $G(z)$ and $\wt{G}(z)$, respectively, for $z \not\in \R$, then
\begin{align}
G(z) - \wt{G}(z) \ = \ \left[ G \left( \wt{H} - H \right) \wt{G} \right](z). \label{eq:GFresolventidentity}
\end{align}
\end{lemma}
As an immediate consequence by letting $\wt{G}(z) = \overline{G}(z)$, we deduce the following off-diagonal averaging identity.
\begin{corollary} \label{corollary:ward}
\emph{(Ward Identity)}

Suppose $H$ is a real symmetric matrix of size $N$ with Green's function $G(z)$. Then for any fixed row index $i \in [[1, N]]$, 
\begin{align}
\sum_{k = 1}^N \ |G_{ik}(E + i \eta)|^2 \ = \ \frac{\on{Im} G_{ii}(E + i \eta)}{\eta}. \label{eq:ward}
\end{align}
In particular, we obtain the following a priori estimate for any matrix index $(i,j)$:
\begin{align}
\left| G_{ij}(E + i \eta) \right| \ \leq \ \frac{1}{\eta}, \label{eq:apGF}
\end{align}
and thus for any matrix index $(i,j)$, the function $G_{ij}(z)$ is locally Lipschitz with constant $\eta^{-2}$.
\end{corollary}
The third preliminary result we give is the following representation of the Green's function $G(z;H)$ in terms of the spectral data of $H$ which is also an important result in compact operator and PDE theory. This spectral representation will be indispensable for exploiting the rich spectral correspondence among covariance matrices and their linearizations.
\begin{lemma} \label{lemma:GFspectral}
\emph{(Spectral Representation)}

Suppose $H$ is a real-symmetric or complex-Hermitian matrix with eigenvalue-eigenvector pairs $\left\{ (\lambda_{\alpha}, \mathbf{u}_{\alpha}) \right\}_{\alpha}$, and let $G(z)$ denote its Green's function. Then for any matrix index $(i,j)$, we have
\begin{align}
G_{ij}(z) \ = \ \sum_{\alpha = 1}^N \ \frac{\mathbf{u}_{\alpha}(i) \overline{\mathbf{u}_{\alpha}}(j)}{\lambda_{\alpha} - z}, \label{eq:GFspectral}
\end{align}
where the overline notation denotes the complex conjugate of the vector entry. In particular, the Green's function is complex Hermitian. 
\end{lemma}
We conclude this preliminary discussion of the Green's function $G(z;H)$ with the following local regularity result concerning a maximal Green's function. The proof of this result may be found as Lemma 2.1 in \cite{BKY}. To state it, we now define the maximal Green's functions of interest, which may be viewed as control parameters for the sake of this paper:
\begin{align}
\Gamma(E + i \eta) \ &= \ \left[ \max_{i,j} \ |G_{ij}(z)| \right] \vee 1, \nonumber \\
\Gamma^{\ast}(E + i \eta) \ &= \ \sup_{\eta' \geq \eta} \Gamma(E + i \eta'). \nonumber
\end{align}
%
\begin{lemma} \label{lemma:GFLipschitzscaling}
For any $z = E + i \eta \in \C_+$, the function $\Gamma(z)$ is locally Lipschitz continuous in $\eta$ with the following bound on its almost-everywhere derivative:
\begin{align}
\left| \partial_{\eta} \Gamma(z) \right| \ \leq \ \frac{\Gamma(z)}{\eta}. \label{eq:GFLipschitz}
\end{align}
In particular, for any $\kappa > 1$ and $z = E + i \eta \in \C_+$, we have
\begin{align}
\Gamma \left( E + i \frac{\eta}{\kappa} \right) \ \leq \ \kappa \Gamma(E + i \eta). \label{eq:GFscaling}
\end{align}
\end{lemma}
\subsection{Reductions of the Proof of Theorem \ref{theorem:locallaw}}
We now return to the setting of biregular bipartite graphs, i.e. the ensembles $\mathscr{X}$, $\mathscr{X}_{\ast}$, and $\mathscr{X}_{\ast,+}$. We begin with the following consequence of Lemma \ref{lemma:GFspectral}, which relates the Green's function entries of matrices from each of the three matrix ensembles of interest.
\begin{lemma} \label{lemma:GFrelations}
Suppose $X$ is a block matrix of the form \eqref{eq:blockmatrix}, and suppose $i,j \in [[1, M+N]]$ are indices chosen such that either $i,j \leq M$ or $i,j > M$. Then, for any $z = E + i \eta \in \C_+$, we have
\begin{align}
G_{ij}(z) \ = \ \begin{cases} 
      z G_{\ast,+}(z^2) & i,j \leq M \\
      z G_{\ast}(z^2) & i, j > M 
   \end{cases}. \label{eq:GFrelations}
\end{align}
\end{lemma}
\begin{proof}
For simplicity, we suppose $X$ is real symmetric as the proof for complex Hermitian matrices is similar. First suppose $i, j \leq M$. By the spectral representation in \eqref{eq:GFspectral} and Proposition \ref{prop:spectralcorrespondence}, we obtain
\begin{align}
G_{ij}(z) \ &= \ \sum_{\alpha} \ \frac{\mathbf{u}_{\alpha}(i) \mathbf{u}_{\alpha}(j)}{\lambda_{\alpha} - z} \ = \ \sum_{\lambda \in \sigma(HH^{\ast})} \ \frac12 \left( \frac{\mathbf{u}_{\alpha}(i) \mathbf{u}_{\alpha}(j)}{\sqrt{\lambda_{\alpha}} - z} \ + \ \frac{\mathbf{u}_{\alpha}(i) \mathbf{u}_{\alpha}(j)}{-\sqrt{\lambda_{\alpha}} - z} \right),
\end{align}
where the last equality holds by abuse of notation for eigenvectors of the covariance matrix $HH^{\ast}$ versus the linearization $X$. This completes the derivation for the case $i, j \leq M$. The proof for the case $i, j > M$ follows by the exact same calculation, but instead taking a summation over $\sigma(H^{\ast} H)$ and noting the eigenvector terms $\mathbf{u}_{\alpha}(i) \mathbf{u}_{\alpha}(j)$ vanish for $\lambda_{\alpha} \in \zeta(X)$ by Statements (V) and (VI) in Proposition \ref{prop:spectralcorrespondence}.
\end{proof}
Lemma \ref{lemma:GFrelations} now gives the first reduction of the proof of Theorem \ref{theorem:locallaw}.
\begin{lemma} \label{lemma:reduction1}
Assuming the setting of Theorem 4.1, then the following two estimates are equivalent:
\begin{itemize}
\item \emph{(I).} For any fixed $\e > 0$, we have with probability at least $1 - e^{-\xi \log \xi}$, uniformly over $z = E + i \eta \in U_{\e}$ with $\eta \gg \xi^2/N$, 
\begin{align}
\max_{i} \left| [G_{\ast}(z)]_{ii} - m_{\infty}(z) \right| \ = \ O \left( F_z(\xi \Phi) \right), \ \ \ \max_{i \neq j} \left| [G_{\ast}(z)]_{ij} \right| \ = \ O(\xi \Phi).
\end{align}
\item \emph{(II).} For any fixed $\e > 0$, we have with probability at least $1 - e^{-\xi \log \xi}$, uniformly over $z = E + i \eta \in U_{\e}$ with $\eta \gg \xi^2/N$,
\begin{align}
\max_{k > M} \left| [G(z)]_{kk} - zm_{\infty}(z^2) \right| \ &= \ O \left( z F_{z^2}(\xi \Phi(z^2)) \right),  \\
\max_{M < k < \ell} \ \left| [G(z)]_{k \ell} \right| \ &= \ O \left(z \xi \Phi(z^2) \right).
\end{align}
\end{itemize}
Similarly, the above equivalence holds replacing $G_{\ast}$ with $G_{\ast,+}$ and taking the maximums over $i \leq M$ and $i,j \leq M$. 
\end{lemma}
From Lemma \ref{lemma:GFrelations}, we also deduce the next reduction.
\begin{lemma} \label{lemma:STlinearequivalence}
Assuming the setting of Theorem 4.1, the following estimates are equivalent for any $z \in \C_+$:
\begin{align}
\left| s_b(z) - zm_{\infty,+}(z^2) \right| \ &= \ O \left( z F_{z^2}(\xi \Phi(z^2)) \right),  \\
\left| s_w(z) - z m_{\infty}(z^2) \right| \ &= \ O \left( z F_{z^2}(\xi \Phi(z^2)) \right). 
\end{align}
Similarly, the following estimates are equivalent for any $z \in \C_+$:
\begin{align}
\left| s_{\ast}(z) - m_{\infty}(z) \right| \ &= \ O \left( F_z(\xi \Phi) \right), \\
\left| s_{\ast,+}(z) - m_{\infty,+}(z) \right| \ &= \ O \left( F_z(\xi \Phi) \right).
\end{align}
\end{lemma}
We briefly note that Lemma \ref{lemma:STlinearequivalence} improves upon Lemma \ref{lemma:reduction1} in that it removes the restriction $|E| > \e$ on the energy.

We are now in a position to make our final reduction of the proof of the local laws in Theorem 4.1, which exploits the first reduction in Lemma \ref{lemma:reduction1} and thus allows us to focus on the covariance matrices $X_{\ast}$ and $X_{\ast,+}$. The reduction will depend on the following result, for which we need to define the following spectral domain.
\begin{align}
\mathbf{D}_{N, \delta, \xi} \ = \ \{ z = E + i \eta: \ |E| \leq N^{\delta}, \ \xi^2/N \ \leq \ \eta \ \leq \ N \}.
\end{align}
%
\begin{prop} \label{prop:selfimprovingestimate}
Suppose $\xi, \zeta > 0$ and $D \gg \xi^2$. If, for a fixed $z \in \mathbf{D}_{N, \delta, \xi}$, we have 
\begin{align}
\mathbb{P} \left( \Gamma^*_{\star}(z) = O(1) \right) \ \geq \ 1 - e^{- \zeta},
\end{align}
then, with probability at least $1 - e^{- (\xi \log \xi) \wedge \zeta + O(\log N)}$, we have
\begin{align}
\max_i \ |[G_{\star}(z)]_{ii} - m(z)| \ = \ O(F_z(\xi \Phi(z))), \ \ \ \max_{i \neq j} \ |[G_{\star}(z)]_{ij}| \ = \ O \left( \frac{\xi \Phi(z^2)}{z} \right). \label{eq:locallawconditioned}
\end{align}
Here, $\star$ can take the values $\star = \ast$ and $\star = \ast, +$.
\end{prop}
To deduce Theorem \ref{theorem:locallaw} from Proposition \ref{prop:selfimprovingestimate}, we follow the exactly the argument used to deduce Theorem 1.1 from Proposition 2.2 in \cite{BKY}. The only thing we need to check to apply the same argument are the following bounds:
\begin{align}
m_{\infty}(z), m_{\infty, +}(z) \ \leq \ C,
\end{align}
for some constant $C = O(1)$. These estimates are proven in the appendix of this paper. We may also extend this argument to remove the energy repulsion assumption $|E| > \e$, which we precisely state in the following lemma.
\begin{lemma} \label{lemma:bootstrapenergy}
Suppose Proposition \ref{prop:selfimprovingestimate} holds. Then the estimates \emph{(4.10)} and \emph{(4.11)} hold without the assumption $|E| > \e$. Consequently, the estimates \emph{(4.14)} and \emph{(4.15)} hold without the assumption $|E| > \e$. Moreover, if $\alpha > 1$, then the estimate \emph{(4.7)} holds without the assumption $|E| > \e$.
\end{lemma}
\begin{proof}
We may again apply the iteration scheme used in proving Theorem 1.1 from Proposition 2.2 in \cite{BKY}. Here, we need to check the estimate $m(z) = O(1)$ and that in the regime $\alpha > 1$, we have the estimate $m_{\infty}(z) = O(1)$. These are similarly derived in the appendix of this paper.
\end{proof}
Thus, contingent on estimates derived in the appendix, to prove Theorem \ref{theorem:locallaw} it suffices to prove Proposition \ref{prop:selfimprovingestimate}. This will be the focus for remainder of the paper. In particular, we may now work with an explicitly smaller domain $\mathbf{D}_{N, \delta, \xi}$ and an a priori estimate on the maximal Green's function.
\subsection{Switchings on Green's Functions}
The main result in this subsection consists of the following estimates comparing Green's function entries to index-wise averages. These estimates will be fundamental to controlling terms that show up naturally in the derivation of the self-consistent equation. Before we can state the result, we need to first introduce a notion of high probability used throughout the remainder of this paper.
\begin{definition}
Fix a parameter $t = t_N \gg \log N$, and a probability space $\Omega$. We say an event $\Xi \subset \Omega$ holds with $t$-\emph{high probability}, or $t$-HP for short, if
\begin{align}
\mathbb{P} \left( \Xi^{C} \right) \ \leq \ e^{-t + O(\log N)}.
\end{align}
\end{definition}
As suggested in Theorem 4.1, we will take the parameter $t = (\xi \log \xi) \wedge \zeta$. We now state the main estimates.
\begin{lemma} \label{lemma:conditionalexpectationscederivation}
Fix a vertex $i = v_b \in V_b$ and an edge label $\mu \in [[1, d_b]]$. Suppose $z = E + i \eta \in \C_+$ satisfies the following constraints for a fixed $\e > 0$:
\begin{align}
|E| \ > \ \e, \quad \eta \ \gtrsim \ \frac{1}{N}. \label{eq:Uelevel}
\end{align}
Suppose further that $\Gamma = O(1)$ holds with $t$-HP. Then for all fixed indices $j, \ell, r$ we have
\begin{align*}
\E_{\mathscr{F}_0} \left( G_{v_{b,\mu} j} - \frac{1}{N} \sum_{k = M+1}^{M+N} \ G_{k j} \right) \ &= \ - \E_{\mathscr{F}_0} \left( d_w^{-1/2} s_w G_{ij} \right) + O(d_w^{-1/2} \Phi), \\
\E_{\mathscr{F}_0} \left[ G_{\ell r}\left( G_{v_{b,\mu} j} - \frac{1}{N} \sum_{k = M+1}^{M+N} \ G_{k j} \right) \right] \ &= \ -\E_{\mathscr{F}_0} \left( G_{\ell r} d_w^{-1/2} s_w G_{ij} \right) + O(d_w^{-1/2} \Phi).
\end{align*}
Similarly, fix a vertex $k = v_w \in V_w$ and $\mu \in [[1, d_w]]$, and suppose $z \in \C_+$ satisfies the constraints \eqref{eq:Uelevel}. Further suppose $\Gamma = O(1)$ with $t$-HP. Then for all fixed indices $j, \ell, r$, we have
\begin{align*}
\E_{\mathscr{F}_0} \left( G_{v_{w,\mu} j} - \frac{1}{M} \sum_{i = 1}^{M} \ G_{i j} \right) \ &= \ - \E_{\mathscr{F}_0} \left( d_w^{-1/2} s_b G_{kj} \right) + O(d_w^{-1/2} \Phi), \\
\E_{\mathscr{F}_0} \left[ G_{\ell r}\left( G_{v_{w,\mu} j} - \frac{1}{M} \sum_{i = 1}^{M} \ G_{i j} \right) \right] \ &= \ -\E_{\mathscr{F}_0} \left( G_{\ell r} d_w^{-1/2} s_b G_{kj} \right) + O(d_w^{-1/2} \Phi).
\end{align*}
\end{lemma}
Before we provide a proof of this result, we will need an auxiliary estimate comparing Green's function entries from i.i.d. samples of graphs. To state this auxiliary estimate, we define first the conditional maximal Green's functions for any fixed edge label $\mu \in [[1, d_b]]$ or $\nu \in [[1, d_w]]$:
\begin{align}
\Gamma_{\mu} \ = \ \Gamma_{\mu}(z) \ := \ \| \Gamma(z) \|_{L^{\infty}(\mathscr{G}_{\mu})}, 
\end{align}
where the notation $L^{\infty}(\mathscr{G}_{\mu})$ in the norm denotes the $L^{\infty}$-norm conditioning on the $\sigma$-algebra $\mathscr{G}_{\mu}$.
\begin{lemma} \label{lemma:preliminaryestimategreensfunction}
For any fixed indices $i, j \in [[1, M+N]]$ and any label $\mu \in [[1, d_b]]$, we have
\begin{align}
G_{ij} - \wt{G}_{ij} \ = \ O(d_w^{-1/2} \Gamma_{\mu} \Gamma).  \label{eq:perturbresample}
\end{align}
Moreover, suppose $x, y$ are random variables such that, conditioned on $\mathscr{G}_\mu$ and $x$, the random variable $y$ is approximately uniform. Then we have
\begin{align}
\E_{\mathscr{G}_{\mu}} \left| G_{xy} \right|^2 \ = \ O(\Gamma_{\mu}^4 \Phi^2). \label{eq:corollaryofperturbresample}
\end{align}
The estimate \label{eq:perturbresample} also holds for any fixed label $\nu \in [[1,d_w]]$ of a white vertex.
\end{lemma}
The proof of this result follows from the proof of Lemma 3.9 in \cite{BKY} so we omit it. We now record the following consequence of Lemma \ref{lemma:preliminaryestimategreensfunction} and proceed to the proof of Lemma \ref{lemma:conditionalexpectationscederivation}.
\begin{corollary} \label{corollary:takemaxperturbgreensfunction}
In the setting of Lemma \ref{lemma:preliminaryestimategreensfunction}, we have
\begin{align}
\Gamma \ = \ \Gamma_\mu \ + \ O \left( d_w^{-1/2} \Gamma_\mu \Gamma \right).
\end{align}
\end{corollary}
\begin{proof}
(of Lemma \ref{lemma:conditionalexpectationscederivation}).

We prove the first estimate only for $v_b \in V_b$; the proof of the second estimate and the estimates for $v_w \in V_w$ are analogous. In expectation, we first note
\begin{align}
\E_{\mathscr{F}_0} G_{v_{b,\mu} j} \ = \ \E_{\mathscr{F}_0} \wt{G}_{\wt{v}_{b,\mu} j}. \nonumber
\end{align}
Moreover, because $G(z)$ is independent of the random variable $\wt{v}_{b,\mu}$, and because, conditioned on $\mathscr{G}_\mu$, the random variable $\wt{v}_{b,\mu} \in V_w$ is approximately uniform, we also know
\begin{align}
\E_{\mathscr{F}_0} \left( \frac{1}{N} \sum_{k = M+1}^{M+N} \ G_{kj} \right) \ = \ \E_{\mathscr{F}_0} G_{\wt{v}_{b,\mu} j} + O(d_b^{-1/2} \Phi). \nonumber
\end{align}
Thus, it suffices to compute 
\begin{align}
- \E_{\mathscr{F}_0} \left( G_{\wt{v}_{b,\mu} j} - \wt{G}_{\wt{v}_{b,\mu} j} \right). \nonumber
\end{align}
By the resolvent identity, we have the following equation holding in expectation:
\begin{align}
\E_{\mathscr{F}_0} \left( - G_{\wt{v}_{b,\mu} j} - \wt{G}_{\wt{v}_{b,\mu} j} \right) \ &= \ \E_{\mathscr{F}_0} \left( \sum_{k, \ell} \ G_{\wt{v}_{b,\mu} k} (\wt{X} - X)_{k \ell} \wt{G}_{\ell j} \right).\label{eq:perturbgreensexpectation}
\end{align}
Unfolding the high-probability equation (11.15) in Lemma 11.9, we have, with probability at least $1 - O(d_w^{-1/2} \Phi)$ conditioned on $\mathscr{F}_0$, 
\begin{align}
\wt{X} - X \ = \ d_w^{-1/2} \left( \Delta_{v_b \wt{v}_{b,\mu}} - \Delta_{v_b v_{b, \mu}} + \Sigma_b \right) \ + \ d_w^{-1/2} \left( \Delta_{\wt{v}_{b,\mu} v_b} - \Delta_{v_{b,\mu} v_b} + \Sigma_b^* \right).\label{eq:perturbresampleexpectation}
\end{align}
Here, we recall that $v_{b,\mu}$ (resp. $\wt{v}_{b,\mu}$) is the vertex adjacent to $v_b$ in $S_{v_b, \mu}$ (resp. $\wt{S}_{v_b, \mu}$) \emph{after} resampling. Also, $\Sigma_b$ is the matrix given by a sum of terms $\pm \Delta_{xy}$ where one of the following two conditions holds:
\begin{itemize}
\item Conditioned on $\mathscr{G}_\mu, \wt{p}_\mu$, the random variable $x$ is approximately uniform, or;
\item Conditioned on $\mathscr{G}_{\mu}$, the random variable $y$ is approximately uniform.
\end{itemize}
Thus, upon unfolding the RHS of \eqref{eq:perturbgreensexpectation}, we see one term is given by, in expectation,
\begin{align}
\E_{\mathscr{F}_0} \left[ d_w^{-1/2} G_{\wt{v}_{b,\mu} \wt{v}_{b,\mu}} \wt{G}_{ij} \right] \ &= \ \E_{\mathscr{F}_0} \ \left[ d_w^{-1/2} \underline{s} \wt{G}_{ij} \right] + O(d_w^{-1/2} \Phi) \nonumber \\
&= \ \E_{\mathscr{F}_0} \ \left[ d_w^{-1/2} \underline{s} G_{ij} \right] + O(d_w^{-1/2} \Phi), \nonumber
\end{align}
where the first equality holds because $\wt{v}_{\beta,\mu}$ is approximately uniform by Corollary 11.6 and the second holds since for any fixed indices $i, j$, we have $G_{ij} \sim \wt{G}_{ij}$ conditioned on $\mathscr{G}_\mu$. In particular, we have
\begin{align}
\E_{\mathscr{F}_0} \left( G_{\wt{v}_{b,\mu} \wt{v}_{b, \mu}} \wt{G}_{ij} \right) \ &= \ \E_{\mathscr{F}_0} \left( G_{\wt{v}_{b,\mu} \wt{v}_{b,\mu}} G_{ij} \right) \ + \ \E_{\mathscr{F}_0} \left[ G_{\wt{v}_{b,\mu} \wt{v}_{b,\mu}} \left( \wt{G}_{ij} - G_{ij} \right) \right] \nonumber \\
&= \ \E_{\mathscr{F}_0} \left( G_{\wt{v}_{b,\mu} \wt{v}_{b,\mu}} G_{ij} \right) + O \left( \frac{1}{\sqrt{D}} \right), \nonumber
\end{align}
where the second equality holds by Lemma \ref{lemma:preliminaryestimategreensfunction}. Thus, it suffices to bound the remaining terms in \eqref{eq:perturbgreensexpectation}. By \eqref{eq:perturbresampleexpectation}, it suffices to estimate the expectation of terms $G_{\wt{v}_{b,\mu} x} \wt{G}_{yj}$. By the second result in Lemma \ref{lemma:preliminaryestimategreensfunction} and the Schwarz inequality, in the case where $y$ is approximately uniform conditioning on $\mathscr{G}_{\mu}$, we have, with high probability,
\begin{align}
\E_{\mathscr{F}_0} G_{\wt{v}_{b,\mu} x} \wt{G}_{yj} \ \leq \ \E_{\mathscr{F}_0} |\wt{G}_{yj}|^2 + O(D^{-1/2}) \ \leq \ O(\Phi). \nonumber
\end{align}
Thus, by the assumption $\Gamma = O(1)$, Lemma \ref{lemma:conditionalexpectationscederivation} follows after accumulating the finitely many events all holding with probability at least $1 - O(d^{-1/2} \Phi)$.
\end{proof}
%
%
%
\section{The Self-Consistent Equation and Proof of Theorem \ref{theorem:locallaw}}
\subsection{Derivation of the self-consistent equation}
We begin by introducing the following two pieces of notation for a random vector $Z = (Z_i)_{i \in [[1, M]]}$ and $\wt{Z} = (\wt{Z}_k)_{k \in [[M+1, M+N]]}$ :
\begin{align}
\E_{(i)} Z \ = \ \frac{1}{M} \sum_{i = 1}^M \ Z_i, \quad \E_{(k)} \wt{Z} \ = \ \frac{1}{N} \sum_{k = M+1}^{M+N} \ \wt{Z}_k.
\end{align}
We now record the following concentration estimate which will allow us to take expectations and exploit the estimates in Lemma {lemma:conditionalexpectationscederivation}. To do so, we introduce the following notation.
\begin{definition}
Suppose $X$ is an $L^1$-random variable, and suppose $\sigma(\cdot)$ is a $\sigma$-algebra which $X$ is measurable with respect to. We define the $\sigma(\cdot)$-\emph{fluctuation} of $X$ to be the following centered random variable:
\begin{align}
X_{\sigma(\cdot)} \ := \ X - \E_{\sigma(\cdot)} X.
\end{align}
\end{definition}
\begin{prop} \label{prop:concentrationestimate}
Suppose that $z = E + i \eta \in \mathbf{D}_{N, \delta, \xi} \cap U_{\e}$, that $\zeta > 0$, and that $\Gamma = O(1)$ with probability at least $1 - e^{-\zeta}$. Fix $k = O(1)$ and pairs of indices $I_k = \{ (i_1, j_1), \ldots, (i_k, j_k)\} $. Define the random variable
\begin{align}
X_{I_k}(z) \ = \ G_{i_1 j_1} \ldots G_{i_k j_k}. \nonumber
\end{align}
Then, for any $\xi = \xi(N)$ satisfying $\xi \to \infty$ as $N \to \infty$, we have the following pointwise concentration estimate:
\begin{align}
\mathbb{P} \left[ \left(X_{I_k}(z) \right)_{\mathscr{F}_0} \ = \ O( \xi \Phi) \right] \ \geq \ 1 - e^{- \left[ (\xi \log \xi) \wedge \zeta \right] + O(\log N)}.
\end{align}
\end{prop}
The proof of Proposition \ref{prop:concentrationestimate} follows from the proof of Proposition 4.1 in \cite{BKY} so we omit it.

We now consider the matrix equation $HG = zG + \on{Id}$ and compute the diagonal entries of both sides. The $(i,i)$-entry of the RHS is clearly given by $zG_{ii} + 1$. We now study the LHS, considering the $(k,k)$-entry for $k \in [[M, M+1]]$. By matrix multiplication we have
\begin{align}
(HG)_{kk} \ = \ \sum_{i = 1}^M \ H_{ki} G_{ik} \ &= \ d_w^{-1/2} \sum_{i = 1}^M \ \sum_{\nu = 1}^{d_w} \left( \delta_{i, v_{b, \nu}} - \frac{1}{M} \right) G_{ik} \\
&= \ d_w^{-1/2} \sum_{\nu = 1}^{d_w} \ \left( G_{v_{b, \nu} k} - \E_{(i)} G_{ik} \right), \label{eq:laststepsce}
\end{align}
where we used the relation $Md_b = Nd_w$. Appealing to Lemma \ref{lemma:conditionalexpectationscederivation} we deduce the following identity:
\begin{align}
\E_{\mathscr{F}_0} (HG)_{kk} \ = \ - \E_{\mathscr{F}_0} \left( \sum_{\nu = 1}^{d_w} \ \left[ \frac{1}{d_w} \ s_b G_{ii} + O(d_w^{-1} \Phi) \right] \right) \ = \ - \E_{\mathscr{F}_0} s_b G_{ii} + O(\Phi).
\end{align}
Taking an expectation conditioning on $\mathscr{F}_0$ in the matrix equation $HG = zG + \on{Id}$, we see
\begin{align}
1 + z \E_{\mathscr{F}_0} G_{kk} \ = \ - \E_{\mathscr{F}_0} s_b G_{kk} + O \left( \Phi \right).
\end{align}
Using Proposition \ref{prop:concentrationestimate} to account for the $\mathscr{F}_0$-fluctuation of the Green's function terms, we ultimately deduce a stability equation for the diagonal $(k,k)$-entries of $G$, with $k > M$. We may run a similar calculation for indices $i \in [[1, M]]$ and derive the following system of equations:
\begin{align}
1 + (z + \gamma s_w) G_{ii} \ &= \ O \left( (1 + |z|) \xi \Phi \right), \\
1 + (z + s_b) G_{kk} \ &= \ O \left((1 + |z|) \xi \Phi \right).
\end{align}
Although this system is a priori coupled, we now appeal to Lemma \ref{lemma:GFrelations} to decouple the equations. More precisely, we deduce the following system of \emph{decoupled} equations:
\begin{align}
1 + \left( z + s_b + \frac{1 - \gamma}{z} \right) G_{ii} \ &= \ O \left( (1 + |z|) \xi \Phi \right), \\
1 + \left( z + \gamma s_w + \frac{\gamma - 1}{z} \right) G_{kk} \ &= \ O \left( (1 + |z|) \xi \Phi \right).
\end{align}
From here, we may proceed in two fashions. First, we may use Lemma 7.3 to deduce stability equations for the Green's functions $G_{\ast}$ and $G_{\ast,+}$, relating the diagonal entries of these Green's functions to the Stieltjes transforms $s_{\ast}$ and $s_{\ast,+}$. On the other hand, we may also average over the diagonal entries and deduce self-consistent equations for the Stieltjes transforms $s_b, s_w$ and $s_{\ast}, s_{\ast,+}$. We summarize these estimates in the following proposition.
\begin{prop} \label{prop:sces}
Suppose $\Gamma = O(1)$ with $t$-HP, and let $z = E + i \eta \in U_{\e}$ satisfy $\eta \gg N^{-1}$. Then for any $i \in [[1, M]]$ and $k \in [[M+1, M+N]]$, we have the following equations uniformly over such $z$ with $t$-HP:
\begin{align}
1 + \left( z + s_b + \frac{1 - \gamma}{z} \right) G_{ii} \ &= \ O((1 + |z|) \xi \Phi),  \label{eq:individualsceupper} \\
1 + \left( z + \gamma s_w + \frac{\gamma - 1}{z} \right) G_{kk} \ &= \ O((1 + |z|) \xi \Phi), \label{eq:individualscelower} \\
1 + \left( z + 1 - \gamma + z s_{\ast,+} \right) [G_{\ast,+}]_{ii} \ &= \ O((1 + |z|^{1/2}) \xi \Phi) \ = \ O((1 + |z|) \xi \Phi), \label{eq:individualscelargecov}  \\
1 + \left( z + \gamma - 1 + \gamma z s_{\ast} \right) [G_{\ast}]_{kk} \ &= \  O((1 + |z|^{1/2}) \xi \Phi) \ = \ O((1 + |z|) \xi \Phi). \label{eq:individualscesmallcov}
\end{align}
Moreover, we have the following averaged equations uniformly over such $z$ with $t$-HP:
\begin{align}
1 + \left( z + s_b + \frac{1 - \gamma}{z} \right) s_b \ &= \ O((1 + |z|) \xi \Phi), \label{eq:averagedsceupper} \\
1 + \left( z + \gamma s_w + \frac{\gamma - 1}{z} \right) s_w \ &= \ O((1 + |z|) \xi \Phi), \label{eq:averagedscelower} \\
1 + \left( z + 1 - \gamma + z s_{\ast,+} \right) s_{\ast,+} \ &= \ O((1 + |z|^{1/2} \xi \Phi) \ = \ O((1 + |z|)\xi \Phi), \label{eq:averagedscelargecov} \\
1 + \left( z + \gamma - 1 + \gamma z s_{\ast} \right) s_{\ast} \ &= \ O((1 + |z|^{1/2}) \xi \Phi) \ = \ O((1 + |z|) \xi \Phi). \label{eq:averagedscesmallcov}
\end{align}
\end{prop}
\begin{proof}
It remains to upgrade the self-consistent equations \eqref{eq:individualsceupper} -- \eqref{eq:averagedscesmallcov} to hold over all such $z = E + i \eta$ with $t$-HP. To this end, we appeal to the Lipschitz continuity of the Green's function entries on a sufficiently dense lattice as in the proof of Lemma 8.5.
\end{proof}
\subsection{Analysis of the self-consistent equation}
From a direct calculation, we note the Stieltjes transform $m_{\infty}$ of the Marchenko-Pastur law with parameter $\gamma \leq 1$ is given by the following self-consistent equation:
\begin{align}
\gamma z m_{\infty}^2(z) + \left( \gamma + z - 1 \right) m_{\infty}(z) + 1 \ = \ 0. \label{eq:sceMP}
\end{align}
For the augmented Stieltjes transform $m_{\infty,+}$, we may similarly deduce a self-consistent equation. In our analysis, we will be concerned with providing full details for the Stieltjes transform $m_{\infty}$ \emph{only}, as the estimate for the augmented transform $m_{\infty,+}$ will follow from the estimate on $m_{\infty}$. We now note Proposition \ref{prop:sces} implies the Stieltjes transform $s_{\ast}$ solves the same self-consistent equation with an error of $o(1)$ throughout the domain $\mathbf{D}_{N, \delta, \xi} \cap U_{\e}$, with $t$-HP. Our goal will be to use the stability of the self-consistent equation \eqref{eq:sceMP} under $o(1)$ perturbations to compare $s_{\ast}$ and $m_{\infty}$. This is the content of the following result.
\begin{prop} \label{prop:stabilityestimateaverage}
Let $m: \C_+ \to \C_+$ be the unique solution to the following equation:
\begin{align}
\gamma zm^2 + (\gamma + z - 1)m + 1 \ = \ 0.
\end{align}
Suppose $s: \C_+ \to \C_+$ is continuous and let 
\begin{align}
R := \gamma zs^2 + (\gamma + z - 1)s + 1.
\end{align}
Fix an energy $E \in \R \setminus [-\varepsilon, \varepsilon]$ for $\e > 0$ small and scales $\eta_0 < C(E)$ and $\eta_{\infty} \leqslant N$, where $C(E) = O_E(1)$ is a constant to be determined. Suppose we have 
\begin{align}
|R(E + i \eta)| \ \leqslant \ (1 + |z|) r(E + i \eta)
\end{align}
for a nonincreasing function $r: [\eta_0, \eta_{\infty}] \to [0,1]$. Then for all $z = E + i \eta$ for $\eta \in [\eta_0, \eta_{\infty}]$, we have the following estimate for sufficiently large $N$:
\begin{align}
| m - s | \ = \ O(F(r)). \label{eq:stabilityestimateaverage}
\end{align}
Here, the constant $C(E)$ is determined by
\begin{align}
\on{Im} \left( \frac{1 - \gamma}{E + i \eta} \right) \ > \ 3 \alpha^{1/2} \varepsilon^{-1/2}
\end{align}
for all $\eta \leqslant C(E)$.
\end{prop}
Before we proceed with the proof of Proposition \ref{prop:stabilityestimateaverage}, we introduce the following notation. 
\begin{notation}
We denote the solutions to the equation \eqref{eq:sceMP} by $m_{\pm}$, where $m_+$ maps the upper-half plane to itself, and $m_-$ maps the upper-half plane to the lower-half plane.

Moreover, we define the following error functions:
\begin{align}
v_{\pm} \ = \ \left| m_{\pm} - s \right|.
\end{align}
\end{notation}
Having established this notation, because $m_{+}$ takes values in the upper-half plane, we deduce the following upper bound on the values taken by the imaginary part of $m_{-}$ as follows:
\begin{align}
\on{Im}(m_-(z)) \ \leqslant \ - \on{Im} \left( \frac{1 - \gamma}{E + i \eta} \right) \ < \ - 3 \alpha^{1/2} \varepsilon^{-1/2}
\end{align}
for scales $\eta < C(E)$. We now proceed to derive an a priori estimate on the error functions $v_{\pm}$. 
\begin{lemma} \label{lemma:easyregimemin}
Under the assumptions and setting of Proposition \ref{prop:stabilityestimateaverage}, we have
\begin{align}
|v_+| \wedge |v_-| \ \leqslant \ 3 \alpha^{1/2} \varepsilon^{-1/2} F(r).
\end{align}
\end{lemma}
\begin{proof}
We appeal to the following inequality which holds for any branch of the complex square root $\sqrt{\cdot}$ and any complex parameters $w, \zeta$ for which the square root is defined:
\begin{align}
| \sqrt{w + \zeta} - \sqrt{w} | \wedge | \sqrt{w + \zeta} + \sqrt{w}| \ &\leqslant \ \frac{|\zeta|}{\sqrt{|w|}} \wedge \sqrt{|\zeta|}.
\end{align}
In particular, this implies the following string of inequalities:
\begin{align}
|v_+| \wedge |v_-| \ &\leqslant \ \frac{1}{2\gamma z} \left( \frac{|4\gamma zR|}{\sqrt{|(\gamma + z - 1)^2 - 4\gamma z|}} \wedge \sqrt{|4 \gamma z R|}\right) \\
&\leqslant \ \frac{2|R|}{\sqrt{|(\gamma + z - 1)^2 - 4 \gamma z|}} \wedge \sqrt{\varepsilon \alpha R} \\
&\leqslant \ \frac{2 \varepsilon^{1/2}(1 + |z|)r(E + i \eta)}{\sqrt{|(\gamma  + z - 1)^2 - 4\gamma z|}} \wedge \sqrt{\varepsilon \alpha  (1 + |z|) r(E + i \eta)}
\end{align}
where the second inequality follows from the assumption $|z| \geqslant |E| \geqslant \varepsilon$ and the last bound follows if we choose $\varepsilon \leqslant 1$. But this is bounded by $3 \alpha^{1/2} \varepsilon^{-1/2} F(r)$ for any $r \in [0,1]$. 
\end{proof}
\begin{proof}
(of Proposition \ref{prop:stabilityestimateaverage}).

We consider two different regimes. First consider the regime where $|m_+ - m_-| > (1 + |z|) r(\eta)$. Precisely, this is the regime defined by $\eta > C(E)$ and the energy-dependent constant $D(E)$ such that
\begin{align}
(1 + |z|) r(\eta) \ < \ \frac{|(\gamma + z - 1)^2 - 4 \gamma z|}{D(E)};
\end{align}
the constant $D(E)$ will be determined later. We note by Lemma \ref{lemma:easyregimemin}, in this regime it suffices to prove the following bound:
\begin{align}
|v_-| \ > \ |v_-| \wedge |v_+|.
\end{align}
We now choose an energy-dependent constant $\kappa(E)$ such that for all $\eta \in [C(E), \eta_{\infty}]$, we have the bound
\begin{align}
\alpha^{1/2} (1+|z|)r(\eta) \ \leqslant \ \kappa(E) (1 + |E + i C(E)|).
\end{align}
Moreover, note $|(\gamma + z - 1)^2 - 4\gamma z|$ is increasing in $\eta$, as seen by translating $z = w + 1 - \gamma$ and computing
\begin{align}
|(\gamma + z - 1)^2 - 4\gamma z| \ = \ |w^2 - 4\gamma w + X| \ = \ |w(w-4 \gamma) + X|, \nonumber
\end{align}
where $X \in \R$. Because $r(\eta)$ is non-increasing in $\eta$, for all $z = E  + i \eta$ with $\eta \in [C(E), \eta_{\infty}]$, we have
\begin{align}
(1+|z|)r(\eta) \ < \ \frac{\kappa(E)}{D(E)} |(\gamma + z - 1)^2 - 4\gamma z|.
\end{align}
We first compute a uniform lower bound on the difference term as follows:
\begin{align}
|v_+ - v_-| \ &= \ \frac{|(\gamma + z - 1)^2 - 4 \gamma z|}{2 \gamma |z|} \nonumber \\
&\geqslant \ \frac{1}{2 \gamma |E + i \eta_{\infty}|} \left( \frac{D(E)}{\kappa(E)} \frac{(1 + |z|) r(\eta)}{\sqrt{|(\gamma + z - 1)^2 - 4\gamma z|}} \ \wedge \ \sqrt{\frac{D(E)}{\kappa(E)}} \sqrt{\alpha (1 + |z|) r(\eta) }\right) \nonumber \\
&> \ 0 \nonumber.
\end{align}
By continuity of $s$ and the estimate Lemma \ref{lemma:easyregimemin}, choosing $D(E)$ large enough as a function of $\kappa(E), E, \eta_{\infty}, \varepsilon$, it suffices to prove the estimate \eqref{eq:stabilityestimateaverage} for some $\eta \in [C(E), \eta_{\infty}]$. But this follows from Lemma \ref{lemma:easyregimemin}; in particular, we have at $\eta = C(E)$ and $N$ sufficiently large,
\begin{align}
|v_-| \ \geqslant \ | \on{Im}(s) - \on{Im}(m_-)| \ \geqslant |\on{Im}(m_-)| \ > \ 3 \varepsilon^{-1/2} \ \geqslant \ 3 \varepsilon^{-1/2} F(r) \ \geqslant \ |v_+| \wedge |v_-|. \nonumber
\end{align}
Thus, we have $|v_+| = |v_+| \wedge |v_-|$ in this first regime, implying the stability estimate \eqref{eq:stabilityestimateaverage}.

Now, we take the regime where the a priori estimate
\begin{align}
|(\gamma + z - 1)^2 - 4\gamma z| \ = \ O((1 + |z|)r(\eta))
\end{align}
holds. Thus, we know
\begin{align}
|v_-| \ &\leqslant \ |v_+| + \frac{\sqrt{|(\gamma + z - 1)^2 - 4 \gamma z|}}{2 \varepsilon} \ = \ |v_+| + O \left( \frac{(1+|z|)r(\eta)}{\sqrt{|(\gamma + z - 1)^2 - 4 \gamma z|}} \wedge \sqrt{(1 + |z|) r} \right) \nonumber \\
&= \ |v_+| + O(F(r)), \nonumber
\end{align}
implying the estimate in the second regime as well.
\end{proof}
We conclude the discussion of the self-consistent equation by noting that Lemma \ref{lemma:GFrelations} allows us to deduce the following local law for the Stieltjes transform $s_{\ast,+}$:
\begin{align}
\left| s_{\ast,+} - m_{\infty,+} \right| \ = \ O(F(r)).
\end{align}
This estimate holds in the regime $z = E + i \eta$ for $\eta \in [\eta_0, \eta_{\infty}]$.

\subsection{Final estimates}
Fix an index $k \in [[1, N]]$, and for notational convenience, \emph{for this calculation only}, we let $G$ denote the Green's function $G_{\ast}$. Consider the following approximate stability equation for the diagonal entry $G_{kk}$:
\begin{align}
1 + (z + \gamma - 1 + \gamma z s_{\ast}) G_{kk} \ = \ O \left( (1 + |z|) \xi \Phi \right).
\end{align}
We now appeal to the following estimate which will allow us to study the stability of this equation upon the replacement $s_{\ast} \to m_{\infty}$ that holds with $t$-HP:
\begin{align}
|z G_{kk}| \ = \ O(1).
\end{align}
Indeed, if $\eta \gg 1$, we have
\begin{align}
\left| z G_{kk} \right| \ \leq \ \left| \frac{E + i \eta}{\eta} \right| \ = \ O(1).
\end{align}
If $E \gg 1$, we appeal to the spectral representation of $G_{kk}$ and deduce
\begin{align}
\left| zG_{kk} \right| \ &\leq \ C \left| \frac{1}{N} \sum_{\lambda} \ \frac{|\mathbf{u}_{\lambda}(k)|^2 \times E}{E - \lambda + i \eta} \right| \\
&\leq \ \frac{C}{N} \sum_{\lambda} \ \left| \frac{E}{E - \lambda + i \eta} \right| \\
&= \ O(1),
\end{align}
where we used the uniform a priori bound $\lambda = O(1)$. If $\eta, E \lesssim 1$, then we appeal to the a priori bound $\Gamma = O(1)$ with $t$-HP and the trivial bound $z = O(1)$. Thus, by Proposition \ref{prop:stabilityestimateaverage}, we have
\begin{align}
1 + (z + \gamma - 1 + \gamma z m_{\infty}) G_{kk} \ = \ O \left((1 + |z|) \xi \Phi \right) + O(F(\xi \Phi)).
\end{align}
On the other hand, the self-consistent equation \eqref{eq:sceMP} implies the Green's function term on the LHS may be written as $-G_{kk}/m_{\infty}$. Moreover, because $m_{\infty} = O(1)$ uniformly on the domain $U_{\e}$, we establish the following estimate that holds over all $z \in \mathbf{D}_{N, \delta, \xi} \cap U_{\e}$ with $t$-HP:
\begin{align}
m_{\infty} - G_{kk} \ = \ O \left( F(\xi \Phi) \right),
\end{align}
where we use the estimate $(1 + |z|) m_{\infty} = O(1)$. This completes the proof of the local law along the diagonal of $G = G_{\ast}$.

To derive the estimate for the off-diagonal entries, we appeal to the Green's function $G(z) = (X - z)^{-1}$ of the linearization $X$. Note this is \emph{no longer} the Green's function $G_{\ast}$ of the covariance matrix $X_{\ast}$. In particular, we appeal to the following entry-wise representation of a matrix equation (for indices $i, j > M$):
\begin{align}
G_{ij}(HG)_{ii} - G_{ii} (HG)_{ij} \ = \ G_{ij}.
\end{align}
As in the derivation of the stability equations in Proposition \ref{prop:sces}, by Lemma \ref{lemma:conditionalexpectationscederivation} the expectation of the LHS is given by
\begin{align}
\E_{\mathcal{F}_0} \left[ G_{ij}d_w^{-1/2} \underline{s} G_{ii} \right] \ - \ \E_{\mathcal{F}_0} \left[ G_{ii} d_w^{-1/2} \underline{s} G_{ij} \right] \ + \ O(\Phi) \ = \ O(\Phi). \nonumber
\end{align}
Thus, at the cost of a concentration estimate in Proposition \ref{prop:concentrationestimate}, we deduce 
\begin{align}
|G_{ij}| \ = \ O(\xi \Phi), 
\end{align}
which yields the estimate for the off-diagonal entries. By Lemma 7.3, this gives the desired estimate for the off-diagonal entries of the Green's function $G_{\ast}$ with $t$-HP. An analogous calculation proves the desired entry-wise estimates for the Green's function $G_{\ast,+}$, which completes the proof of Theorem 4.1.
%
%
%
\section{Appendix}
\subsection{Estimates on the Stieltjes transforms on the upper-half plane}
Here, we want to control the growth of the Stieltjes transforms on the upper-half plane. This is summarized in the following lemma.
\begin{lemma} \label{lemma:STMPlinearbound}
Uniformly over $z \in \C_{+}$, we have 
\begin{align}
m(z) \ = \ O(1). \label{eq:STMPlinearbound}
\end{align}
\end{lemma}
We proceed by considering the two regimes $\gamma = 1$ and $\gamma < 1$, which we refer to as the \emph{square regime} and \emph{rectangular regime}, respectively. \newline

\noindent {\bf{Square Regime.}}
If $\gamma = 1$, we rewrite the density $\varrho(E)$ as
\begin{align}
\varrho_{\gamma = 1}(E) \ = \ \frac{\sqrt{4 - E^2}}{2 \pi} \mathbf{1}_{E \in [-2, 2]},
\end{align}
which is the well-studied semicircle density, whose Stieltjes transform is given by
\begin{align}
m(z) \ = \ \frac{-z + \sqrt{z^2 - 4}}{2}.
\end{align}
For a reference on the semicircle law and its Stieltjes transform, we cite \cite{AGZ}, \cite{BHKY}, \cite{BKY}, \cite{ESY}, \cite{EYY}, and \cite{LY}. We note the branch of the square root is taken so that $\sqrt{z^2 - 4} \sim z$ for large $z$, in which case the bound \eqref{eq:STMPlinearbound} follows immediately. \newline

\noindent {\bf{Rectangular Regime.}}
Fix constants $\Lambda > 0$ and $\e > 0$ to be determined. Suppose $|E| \in [\e, \Lambda]$. By the representation $m(z) = zm_{\infty}(z^2)$ and the explicit formula for $m_{\infty}(z)$ as given in \eqref{eq:STMP}, the bound \eqref{eq:STMPlinearbound} follows immediately in this energy regime, where the implied constant in \eqref{eq:STMPlinearbound} may be taken independent of $\eta$.

Suppose now that $|E| > \Lambda$. Again by the representation $m(z) = zm_{\infty}(z^2)$, we have
\begin{align}
|m(z)| \ &= \ O \left( -z^2 + i \sqrt{(\lambda_{+} - z^2)(z^2 - \lambda_{-})}\right) \\
&= \ O \left(-z^2 + \sqrt{(z^2 - \lambda_{+})(z^2 - \lambda_{-})} \right) \\
&= \ O(1),
\end{align}
since the square root is, again, chosen so that $\sqrt{z^4 + O(z^2)} \sim z^2$ for large $z$. 

Lastly, suppose $|E| < \e$. By definition of $m(z)$ as the Stieltjes transform of $\varrho$, we obtain
\begin{align}
|m(z)| \ = \ \left| \int_{E^2 \in [\lambda_{\pm}]} \ \frac{\gamma}{(1 + \gamma) \pi |E| (E - \e)} \sqrt{(\lambda_+ - E^2)(E^2 - \lambda_-)} \ dE \right| \ = \ O \left(\frac{1}{\sqrt{\lambda_{-} - \e}} \right),
\end{align}
where the implied constant depends only on fixed data $\gamma, \lambda_{\pm}$. Choosing $\e = \sqrt{\lambda_{-}}/100 > 0$, we obtain the desired bound. We note this choice of $\e$ is positive if and only if $\alpha > 1$. This completes the proof of Lemma \ref{lemma:STMPlinearbound}. \QED
%
%
%

\end{document}